\documentclass[reqno]{amsart}
\usepackage{amsmath, amsfonts, amsthm, amssymb, parskip, setspace, textcomp,
	fullpage, bbm}
\newtheorem{theorem}{Theorem}[section]

\newtheorem{proposition}{Proposition}
\newtheorem{conjecture}{Conjecture}

\theoremstyle{definition}

\newtheorem{example}[theorem]{Example}

\newtheorem{bij}[theorem]{Bijection}

\theoremstyle{remark}
\newtheorem{remark}[theorem]{Remark}

\numberwithin{equation}{section}

\newcommand{\spt}[1]{\mbox{\normalfont spt}\Parans{#1}}
\newcommand{\sptBar}[2]{\overline{\mbox{\normalfont spt}}_{#1}\Parans{#2}}
\newcommand{\Mspt}[1]{\mbox{\normalfont M2spt}\Parans{#1}}

\newcommand{\Parans}[1]{\left(#1\right)}
\newcommand{\CBrackets}[1]{\left\{#1\right\}}

\newcommand{\PieceTwo}[4]
{
	\left\{
   	\begin{array}{ll}
      	#1 & #3 \\
       	#2 & #4
     	\end{array}
	\right.
}
\newcommand{\aqprod}[3]{\Parans{#1;#2}_{#3}}
\newcommand{\Jac}[2]{\left(\frac{#1}{#2}\right)}
\newcommand{\GEta}[3]{\eta_{#1,#2}\Parans{#3}}
\newcommand{\STwoB}{\mbox{\rm S2}}

\newcommand{\SPB}{\overline{\mbox{\rm SP}}}
\newcommand{\SB}{\overline{\mbox{\rm S}}}
\newcommand{\sptcrank}{\overline{\mbox{\rm sptcrank}}}
\newcommand{\crankb}{\overline{\mbox{\rm crank}}}
\newcommand{\sptb}{\overline{\mbox{\rm spt}}}
\newcommand{\qbin}[2]{\left[\begin{matrix} #1 \\ #2 \end{matrix} \right]}
\newcommand{\kb}{\overline{k}}
\newcommand{\ov}[1]{\overline{#1}}
\newcommand{\dd}{\mbox{--}}
\newcommand\ktwid{\overset {\text{\lower 3pt\hbox{$\sim$}}}k}
\newcommand\lambdatwid{\overset {\text{\lower 3pt\hbox{$\sim$}}}\lambda}
\newcommand{\abs}[1]{\lvert#1\rvert}
\providecommand{\spt}{\mbox{\rm spt}}

\newcommand\mylabel[1]{\label{#1}}
\newcommand{\beqs}{\begin{equation*}}
\newcommand{\eeqs}{\end{equation*}}
\newcommand{\beq}{\begin{equation}}
\newcommand{\eeq}{\end{equation}}
\newcommand\eqn[1]{(\ref{eq:#1})}
\newcommand\thm[1]{\ref{thm:#1}}

\author{FRANK G. GARVAN}
\address{Department of Mathematics, University of Florida\\
Gainesville, Florida 32611, USA
\endgraf  fgarvan@ufl.edu}

\author{CHRIS JENNINGS-SHAFFER}
\address{Department of Mathematics, University of Florida\\
Gainesville, Florida 32611, USA
\endgraf cjenningsshaffer@ufl.edu}

\keywords{Number theory; Andrews' spt-function; congruences; partitions; smallest parts function.}

\subjclass[2010]{Primary 11P83, 05A17, 11P82, 05A19}

\title{The SPT-Crank for Overpartitions}

\allowdisplaybreaks
\begin{document}

\begin{abstract}
Bringmann, Lovejoy, and Osburn \cite{BLO2,BLO1} showed that the generating 
functions of the spt-overpartition functions $\sptBar{}{n}$, $\sptBar{1}{n}$, $\sptBar{2}{n}$, and 
$\Mspt{n}$ are quasimock theta functions, and satisfy a number of simple 
Ramanujan-like congruences. Andrews, Garvan, and Liang \cite{AGL} defined an 
spt-crank in terms of weighted vector partitions which combinatorially explain
simple congruences mod $5$ and $7$ for $\spt{n}$. Chen, Ji, and Zang \cite{Ch-Ji-Za}
were able to define this spt-crank in terms of ordinary partitions. In this 
paper we define spt-cranks in terms of vector partitions that combinatorially
explain the known simple congruences for all the spt-overpartition functions
as well as new simple congruences. 
For all the overpartition functions except $\Mspt{n}$ we are able to define 
the spt-crank purely in terms of marked overpartitions.
The proofs of the congruences depend on Bailey's Lemma and the difference 
formulas for the Dyson rank of an overpartition \cite{LO1} 
and the $M_2$-rank of a partition without repeated odd parts
\cite{LO2}.
\end{abstract}

\maketitle

\section{Introduction}

\allowdisplaybreaks

Here we consider Ramanujan type congruences for various spt type 
functions and combinatorial interpretations of them in terms of rank and crank 
type functions. We recall the spt function began
with Andrews in \cite{Andrews} defining $\spt{n}$ as the number of 
smallest parts in the partitions of $n$. 
In the same paper he proved the following congruences.
\begin{theorem}\label{TheoremSptCongruences}
For $n\ge 0$ we have
\begin{align}\label{TheoremSptCongruence5np4}
	\spt{5n+4} &\equiv 0 \pmod{5},
\end{align}
\begin{align}\label{TheoremSptCongruence7np5}
	\spt{7n+5} &\equiv 0 \pmod{7},
\end{align}
\begin{align}\label{TheoremSptCongruence13np6}
	\spt{13n+6} &\equiv 0 \pmod{13}.
\end{align}
\end{theorem}
These congruences are reminiscent of the Ramanujan congruences
for the partition function. The proof of Theorem \ref{TheoremSptCongruences} 
relied on relating the $\mbox{\normalfont spt}$ function
to the second moment of the rank function for partitions. With this
$\spt{n}$ could be expressed in terms of rank differences. Formulas for
the required rank differences are found in \cite{AS} and 
\cite{Obrien}.

We recall an overpartition of $n$ is a partition of $n$ in which the first
occurrence of a part may be overlined.  In \cite{BLO2} Bringmann, Lovejoy, and Osburn defined 
 $\sptBar{}{n}$ as the number of 
smallest parts in the overpartitions of $n$. Additionally they defined
$\sptBar{1}{n}$ to be the number of smallest parts in the overpartitions
of $n$ with smallest part odd and $\sptBar{2}{n}$ to be the number of 
smallest parts in the overpartitions of $n$ with smallest even. 
We alter these definitions to only count the smallest parts of the
overpartitions of $n$ where the smallest part is not overlined.
This simply means the count of smallest parts here is half of 
the count of smallest parts in \cite{BLO2} and in other articles. This does not 
have any affect on congruences unless the modulus is even. We
illustrate this change with an example.

The overpartitions of $4$ are
$$
4, 
\mbox{ }\overline{4}, 
\mbox{ }3+1, 
\mbox{ }\overline{3}+1, 
\mbox{ }3+\overline{1}, 
\mbox{ }\overline{3}+\overline{1},
\mbox{ }2+2, 
\mbox{ }\overline{2}+2, 
\mbox{ }2+1+1, 
\mbox{ }\overline{2}+1+1, 
\mbox{ }2+\overline{1}+1, 
\mbox{ }\overline{2}+\overline{1}+1, 
\mbox{ }1+1+1+1, 
\mbox{ }\overline{1}+1+1+1
$$
and so $\sptBar{}{4} = 13$, $\sptBar{1}{4} = 10$, and $\sptBar{2}{4} = 3$.

Bringmann, Lovejoy, and Osburn \cite{BLO2} proved the following congruences
for these new spt functions.
\begin{theorem}\label{TheoremSptBarCongruences}
For $n\ge 0$ we have
\begin{align}\label{TheoremSptBarCongruence3n}
	\sptBar{}{3n} & \equiv 0 \pmod{3},
	\\
	\label{TheoremSptBar1Congruence3n}
	\sptBar{1}{3n} &\equiv 0 \pmod{3},
	\\
	\label{TheoremSptBar1Congruence5n}
	\sptBar{1}{5n} &\equiv 0 \pmod{5},
	\\
	\label{TheoremSptBar2Congruence3n}
	\sptBar{2}{3n} &\equiv 0 \pmod{3},
	\\
	\label{TheoremSptBar2Congruence3np1}
	\sptBar{2}{3n+1} &\equiv 0 \pmod{3},
	\\
	\label{TheoremSptBar2Congruence5np3}
	\sptBar{2}{5n+3} &\equiv 0 \pmod{5}.
\end{align}
\end{theorem}
The proof of these congruences relied on expressing these functions 
in terms of the second moments of certain rank and crank functions which
relate to quasi-modular forms. We will give another proof of
these congruences which gives new combinatorial interpretations of these 
congruences. We describe this method shortly.

In \cite{ABL} Ahlgren, Bringmann, and Lovejoy defined $\Mspt{n}$ 
to be the number of smallest parts in the partitions of $n$
without repeated odd parts and with smallest part even. One congruence
they proved for $\Mspt{n}$ is that for any prime $\ell\ge 3$,
any integer $m\ge 1$, and $n$ such that $\Jac{-n}{\ell}=1$, we have
\begin{align}
	\Mspt{\frac{\ell^{2m}n+1}{8}}\equiv 0\pmod{\ell^m}.
\end{align}
However none of the current known congruences for $\Mspt{n}$ appear to be of 
the form of the congruences we
have mentioned for $\spt{n}$ and $\sptBar{}{n}$, rather they are congruences
related to certain Hecke operators. One of the results of
this paper will be to prove such congruences by giving combinatorial 
refinements.

We will prove the following congruences for $\Mspt{n}$.
\begin{theorem}\label{TheoremM2SptBarCongruences}
For $n\ge 0$ we have
\begin{align}\label{TheoremM2SptBarCongruence3np1}
	\Mspt{3n+1} &\equiv 0 \pmod{3},
\end{align}
\begin{align}\label{TheoremM2SptBarCongruence5np1}
	\Mspt{5n+1} &\equiv 0 \pmod{5},
\end{align}
\begin{align}\label{TheoremM2SptBarCongruence5np3}
	\Mspt{5n+3} &\equiv 0 \pmod{5}.
\end{align}
\end{theorem}

Also we will determine the parity of $\sptBar{}{n}$, $\sptBar{1}{n}$, 
and $\sptBar{2}{n}$.
\begin{theorem}\label{TheoremSptBar1CongruenceOdd}
For $n\ge 1$ we have
	$\sptBar{}{n} \equiv 1 \pmod{2}$
if and only if $n$ is a square or twice a square,
	$\sptBar{1}{n} \equiv 1 \pmod{2}$
if and only if $n$ is an odd square, and
	$\sptBar{2}{n} \equiv 1 \pmod{2}$
if and only if $n$ is an even square or twice a square.
\end{theorem}
In Theorem \ref{TheoremSptBar1CongruenceOdd} it is important 
to note that we're using the convention
of not counting the smallest parts of overpartitions when
the smallest part is overlined. Otherwise we would have $\sptBar{}{n}$,
$\sptBar{1}{n}$, and $\sptBar{2}{n}$ are 
trivially always even and instead these congruences tells
when they are $0$ or $2$ modulo $4$. The method we use to prove these
parity results gives a combinatorial explanation as well, however if one works
modulo $2$ just with the single variable generating functions listed below, 
the parity follows immediately upon noticing the generating functions reduce 
to the sum of divisors generating function.

The generating functions for the spt functions are given as follows, these are special
cases of a general SPT function due to Bringmann, Lovejoy, and Osburn
\cite[Section 7]{BLO1},
\begin{align}
	\mbox{\normalfont SPT}(d,e;q)
	&=
	\frac{\aqprod{-dq}{q}{\infty}\aqprod{-eq}{q}{\infty}}
	{\aqprod{deq}{q}{\infty}\aqprod{q}{q}{\infty}}
	\sum_{n=1}^\infty	
	\frac{q^n\aqprod{q}{q}{n}\aqprod{deq}{q}{n}}
	{(1-q^n)^2 \aqprod{-dq}{q}{n} \aqprod{-eq}{q}{n} }.
\end{align}

The case $d=0$, $e=0$ gives a generating function for $\spt{n}$,
\begin{align}
	\sum_{n=1}^\infty \spt{n}q^n	
	&=
	\sum_{n=1}^\infty \frac{q^n}
		{(1-q^n)^2\aqprod{q^{n+1}}{q}{\infty}}.
\end{align}

The case $d=1$, $e=0$ gives a generating function for $\sptBar{}{n}$, 
\begin{align}\label{Section1SptBar}
	\sum_{n=1}^\infty \sptBar{}{n}q^n	
	&=
	\sum_{n=1}^\infty \frac{q^n \aqprod{-q^{n+1}}{q}{\infty}}
		{(1-q^n)^2\aqprod{q^{n+1}}{q}{\infty}}.
\end{align}

The case $d=1$, $e=1/q, q=q^2$ gives a generating function for $\sptBar{2}{n}$, 
\begin{align}\label{Section1Spt2Bar}
	\sum_{n=1}^\infty \sptBar{2}{n}q^n	
	&=
	\sum_{n=1}^\infty \frac{q^{2n} \aqprod{-q^{2n+1}}{q}{\infty}}
		{(1-q^{2n})^2\aqprod{q^{2n+1}}{q}{\infty}}.
\end{align}

Similar to the $\sptBar{}{n}$ and $\sptBar{2}{n}$ we see a generating function for $\sptBar{1}{n}$ is
\begin{align}\label{Section1Spt1Bar}
	\sum_{n=1}^\infty \sptBar{1}{n}q^n	
	&=
	\sum_{n=0}^\infty \frac{q^{2n+1} \aqprod{-q^{2n+2}}{q}{\infty}}
		{(1-q^{2n+1})^2\aqprod{q^{2n+2}}{q}{\infty}}.
\end{align}

The case $d=0$, $e=1/q, q=q^2$ gives a generating function for $\Mspt{n}$,
\begin{align}\label{Section1M2Spt}
	\sum_{n=1}^\infty \Mspt{n}q^n
	&=
	\sum_{n=1}^\infty \frac{q^{2n} \aqprod{-q^{2n+1}}{q^2}{\infty}}
		{(1-q^{2n})^2\aqprod{q^{2n+2}}{q^2}{\infty}}.
\end{align}
Here we are using the product notation,
\begin{align}
	\aqprod{a}{q}{\infty} &= \prod_{k=0}^\infty (1-aq^k)
	,
	\\
	\aqprod{a}{q}{n} &= \frac{\aqprod{a}{q}{\infty}}{\aqprod{aq^n}{q}{\infty}}
	,
	\\
	\aqprod{a_1,a_2,\dots,a_j}{q}{\infty} 
	&=
	\aqprod{a_1}{q}{\infty}\aqprod{a_2}{q}{\infty}\dots\aqprod{a_j}{q}{\infty}		
	,
	\\
	\aqprod{a_1,a_2,\dots,a_j}{q}{n} 
	&=
	\aqprod{a_1}{q}{n}\aqprod{a_2}{q}{n}\dots\aqprod{a_j}{q}{n}		
	.
\end{align}

We note that the three special cases of SPT$(d,e;q)$ described above
are quasimock theta functions. See \cite[page 240]{BLO1} for a definition.

Andrews, Garvan, and Liang \cite{AGL} found combinatorial interpretations
of the mod $5$ and $7$ congruences in Theorem \ref{TheoremSptCongruences} 
in terms of weighted counts of special vector partitions called 
S-partitions. This was done by adding an extra variable to the generating
function of the spt-function. In particular they defined
\begin{align}
	\mbox{\rm S}(z,q)
	&=
	\sum_{n=1}^\infty \frac{q^n \aqprod{q^{n+1}}{q}{\infty}}
		{\aqprod{zq^n}{q}{\infty}\aqprod{z^{-1}q^n}{q}{\infty}}
	\\
	&= \sum_{n=1}^\infty\sum_{m=-\infty}^\infty
		N_S(m,n)z^mq^n.
\end{align}
One then finds the congruences in (\ref{TheoremSptCongruence5np4})
and (\ref{TheoremSptCongruence7np5}) follow by showing the coefficients
of $q^{5n+4}$ in S$(\zeta_5,q)$ and $q^{7n+5}$ in S$(\zeta_7,q)$
are zero, where $\zeta_5$ is a primitive fifth root of unity and
$\zeta_7$ is a primitive seventh root of unity. This is the approach we 
take to prove the congruences for 
$\sptBar{}{n}$, $\sptBar{1}{n}$, $\sptBar{2}{n}$, and $\Mspt{n}$,
and their combinatorial refinements.

In the next section we give two variable generalizations of the generating
functions (\ref{Section1SptBar}) -- (\ref{Section1M2Spt}), introduce various
ranks and cranks, and state numerous
identities for these functions. At the end of the next section we describe
the plan for the remainder of the paper.

\section{Statement of Results and Preliminaries}\label{sec:statementOfResults}

In this paper we give alternate proofs of the congruences in Theorem 
\ref{TheoremSptBarCongruences} and prove the congruences
of Theorems \ref{TheoremM2SptBarCongruences} and 
\ref{TheoremSptBar1CongruenceOdd} as well as giving combinatorial 
interpretations. We consider two variable 
generalizations of the generating functions 
from the introduction. We set
\begin{align}
 	\SB(z,q)
	&= \sum_{n=1}^\infty 
			\frac{q^n \aqprod{-q^{n+1}}{q}{\infty}\aqprod{q^{n+1}}{q}{\infty}}
			{\aqprod{zq^n}{q}{\infty}\aqprod{z^{-1}q^n}{q}{\infty}} 
	\\
	&= \sum_{n=1}^\infty\sum_{m=-\infty}^\infty N_{\SB}(m,n)z^mq^n
	,
	\\
	\SB_2(z,q)
	&=
	\sum_{n=1}^\infty
		\frac{q^{2n}\aqprod{-q^{2n+1}}{q}{\infty}\aqprod{q^{2n+1}}{q}{\infty}}
			{\aqprod{zq^{2n}}{q}{\infty}\aqprod{z^{-1}q^{2n}}{q}{\infty}}
	\\
	&=
	\sum_{n=1}^\infty\sum_{m=-\infty}^\infty N_{\SB_2}(m,n)z^mq^n
	,
	\\
	\SB_1(z,q)
	&= \sum_{n=0}^\infty 
			\frac{q^{2n+1}\aqprod{-q^{2n+2}}{q}{\infty}\aqprod{q^{2n+2}}{q}{\infty}}
			{\aqprod{zq^{2n+1}}{q}{\infty}\aqprod{z^{-1}q^{2n+1}}{q}{\infty}}
	\\
	&=
		\sum_{n=1}^\infty\sum_{m=-\infty}^\infty N_{\SB_1}(m,n)z^mq^n
	,
	\\
   \STwoB(z,q)
	&=\sum_{n=1}^\infty 
		\frac{q^{2n}\aqprod{q^{2n+2}}{q^2}{\infty}\aqprod{-q^{2n+1}}{q^2}{\infty}}
		{\aqprod{zq^{2n}}{q^2}{\infty}\aqprod{z^{-1}q^{2n}}{q^2}{\infty}}
	\\
	&= \sum_{n=1}^\infty\sum_{m=-\infty}^\infty N_{\STwoB}(m,n)z^mq^n
.
\end{align}

In each two variable generating function we set $z=1$ to recover the generating
functions from the introduction. We see that $\SB(1,q)$ is the 
generating function for $\sptBar{}{n}$, $\SB_2(1,q)$ is the generating function 
for $\sptBar{2}{n}$, $\SB_1(1,q)$ is the generating function for 
$\sptBar{1}{n}$, and $\STwoB(1,q)$ is the generating function for $\Mspt{n}$.

Furthermore we define
\begin{align}
	N_{\SB}(k,t,n) = \sum_{m\equiv k \pmod{t}}N_{\SB}(m,n)
\end{align}
so
\begin{align}
	\sptBar{}{n} &= \sum_{m=-\infty}^\infty N_{\SB}(m,n)
	= \sum_{k=0}^{r-1} N_{\SB}(k,r,n)
\end{align}
for any positive integer $r$.
We similarly define 
\begin{align}
	N_{\SB_1}(k,t,n)
	&=	
	\sum_{m\equiv k \pmod{t}}N_{\SB_1}(m,n)
	,
	\\
	N_{\SB_2}(k,t,n)
	&=
	\sum_{m\equiv k \pmod{t}}N_{\SB_2}(m,n)
	,
	\\
	N_{\STwoB}(k,t,n)
	&=
	\sum_{m\equiv k \pmod{t}}N_{\STwoB}(m,n)
.
\end{align}

We use these series to give another proof of the spt congruences.

First we consider the congruence in (\ref{TheoremSptBarCongruence3n}) 
of Theorem \ref{TheoremSptBarCongruences}. With $\zeta_3$ a primitive third root of unity, we have
\begin{align*}
	\SB(\zeta_3,q)
	&= \sum_{n=1}^\infty\Parans{
			N_{\SB}(0,3,n) 
			+N_{\SB}(1,3,n)\zeta_3
			+N_{\SB}(2,3,n)\zeta_3^2}q^n
\end{align*}

The minimal polynomial for $\zeta_3$ is $1+x+x^2$, and so if
\begin{align*}
	N_{\SB}(0,3,n) 
	+N_{\SB}(1,3,n)\zeta_3
	+N_{\SB}(2,3,n)\zeta_3^2
	&= 0
\end{align*}
then 
\begin{align}\label{proofExplanation1}
	N_{\SB}(0,3,3n)&=N_{\SB}(1,3,3n)=N_{\SB}(2,3,3n).
\end{align}
But if (\ref{proofExplanation1}) holds, then 
\begin{align}
	\sptBar{}{3n} &= 3N_{\SB}(k,3,3n) \hspace{20pt}\mbox{ for } k=0,1,2
\end{align}
and so clearly $\sptBar{}{3n}\equiv 0\pmod{3}$.
That is to say, if we show the coefficient of $q^{3n}$ in $\SB(\zeta_3,q)$ 
to be zero, then we have proved the first congruence in Theorem 
\ref{TheoremSptBarCongruences}, and the stronger result (\ref{proofExplanation1}).

In the same fashion, the congruences 
(\ref{TheoremSptBar1Congruence3n}) and (\ref{TheoremSptBar1Congruence5n}),
will follow by showing the coefficients of $q^{3n}$
in $\SB_1(\zeta_3,q)$ and the coefficients of $q^{5n}$
in $\SB_1(\zeta_5,q)$ are zero. The congruences
(\ref{TheoremSptBar2Congruence3n}), (\ref{TheoremSptBar2Congruence3np1}),
and (\ref{TheoremSptBar2Congruence5np3}) will follow by
showing the coefficients of $q^{3n}$ and $q^{3n+1}$ in $\SB_2(\zeta_3,q)$
and the coefficients of $q^{5n+3}$ in $\SB_2(\zeta_5,q)$ are zero.
The congruences in Theorem \ref{TheoremM2SptBarCongruences} will follow by showing 
the coefficients of $q^{3n+1}$ in $\STwoB(\zeta_3,q)$ and the coefficients 
of $q^{5n+1}$ and $q^{5n+3}$ in $\STwoB(\zeta_5,q)$ are zero.

To this end, we will express the series $\SB(z,q)$, $\SB_1(z,q)$,
$\SB_2(z,q)$, $\STwoB(z,q)$ as the difference of the generating functions for 
certain ranks and cranks. In \cite{AGL} Andrews, the first author, and Liang found that 
S$(z,q)$ could be expressed in terms of the difference of the rank and crank of
a partition. We recall the rank of a partition is the largest part minus the number 
of parts. The crank of a partition is the largest part if there are no ones and
otherwise is the number of parts larger than the number of ones 
minus the number of ones.

As in \cite{BLO2}, for an overpartition $\pi$ of $n$ we define a residual crank 
of $\pi$ by the crank of the subpartition of $\pi$ consisting of the non-overlined 
parts of $\pi$. We let $\overline{M}(m,n)$ denote the number of overpartitions of 
$n$ with this residual crank equal to $m$. The generating function for 
$\overline{M}(m,n)$ is then given by
\begin{align}\label{OverpartitionResidualCrankDef}
	\sum_{n=0}^\infty\sum_{m=-\infty}^\infty \overline{M}(m,n)z^mq^n
	&= \frac{\aqprod{-q}{q}{\infty}\aqprod{q}{q}{\infty}}
		{\aqprod{zq}{q}{\infty}\aqprod{z^{-1}q}{q}{\infty}}.
\end{align}
Of course this interpretation is not quite correct, as  
$\frac{\aqprod{q}{q}{\infty}}{\aqprod{zq,z^{-1}q}{q}{\infty}}$
does not agree at $q^1$ for the crank of the partition consisting of a single one. Thus
the interpretation of this residual crank is not quite correct for overparititons
whose non-overlined part consists of a single one.

As in \cite{BLO2} and others, for an overpartition $\pi$ of $n$ we define the 
Dyson rank of $\pi$ to be the largest part minus the number of parts of $\pi$. 
Let $\overline{N}(m,n)$ denote the number of overpartitions of $n$ with Dyson rank 
equal to $m$. As in Proposition 1.1 and the proof of Proposition 3.2 of 
\cite{Lovejoy1}, the generating function for $\overline{N}(m,n)$ is given by
\begin{align}
	\sum_{n=0}^\infty\sum_{m=-\infty}^\infty \overline{N}(m,n)z^mq^n
	&= \sum_{n=0}^\infty\frac{\aqprod{-1}{q}{n}q^{n(n+1)/2}}
		{\aqprod{zq}{q}{n}\aqprod{z^{-1}q}{q}{n}}
	\\\label{OverpartitionDysonRank}
	&= \frac{\aqprod{-q}{q}{\infty}}{\aqprod{q}{q}{\infty}}
		\Parans{1+
			2\sum_{n=1}^\infty\frac{(1-z)(1-z^{-1})(-1)^nq^{n^2+n}}
				{(1-zq^n)(1-z^{-1}q^n)}
		}
.
\end{align}
The second equality is obtained by Watson's transformation.

We define another residual crank as follows.
For a partitions $\pi$ of $n$ with distinct odd parts we take 
the crank of the partition $\frac{\pi_e}{2}$ obtained by taking the 
subpartition $\pi_e$, of the even parts of 
$\pi$, and halving each part of $\pi_e$.
We let $M2(m,n)$ denote the number of partitions $\pi$ of $n$ with distinct odd parts 
and such that the partition $\frac{\pi_e}{2}$ has crank $m$. 
Then the generating function for $M2$ is given by
\begin{align}\label{M2Crank}
	\sum_{n=0}^\infty\sum_{m=-\infty}^\infty M2(m,n)z^mq^n
	&= 
		\frac{\aqprod{-q}{q^2}{\infty}\aqprod{q^2}{q^2}{\infty}}
			{\aqprod{zq^2}{q^2}{\infty}\aqprod{z^{-1}q^2}{q^2}{\infty}}.
\end{align}
Again this interpretation is not quite correct, here it fails for partitions 
with distinct odd parts whose only even parts are a single two.

We recall the $M_2$-rank of a partition $\pi$ without repeated odd
parts is given by
\begin{align}
	M_2\mbox{-rank} = \left\lceil\frac{l(\pi)}{2} \right\rceil - \#(\pi),
\end{align}
where $l(\pi)$ is the largest part of $\pi$ and $\#(\pi)$ is the number of
parts of $\pi$. The $M_2$-rank was introduced by Berkovich and the first author
in \cite{BG2}. We let $N2(m,n)$ denote the number of partitions of $n$ with
distinct odd parts 
and $M_2$-rank $m$. By Lovejoy and Osburn \cite{LO2} the generating function 
for $N2$ is given by
\begin{align}
	\sum_{n=0}^\infty\sum_{m=-\infty}^\infty N2(m,n)z^mq^n
	&= \sum_{n=0}^\infty q^{n^2}
		\frac{\aqprod{-q}{q^2}{n}}
			{\aqprod{zq^2}{q^2}{n}\aqprod{z^{-1}q^2}{q^2}{n}}.
\end{align}

We set
\begin{align}
	\overline{N}(k,t,n)
	&=	
	\sum_{m\equiv k \pmod{t}}\overline{N}(m,n)
	,
	\\
	\overline{M}(k,t,n)
	&=	
	\sum_{m\equiv k \pmod{t}}\overline{M}(m,n)
	,
	\\
	N2(k,t,n)
	&=
	\sum_{m\equiv k \pmod{t}}N2(m,n)
	,
	\\
	M2(k,t,n)
	&=
	\sum_{m\equiv k \pmod{t}}M2(m,n)
	.
\end{align}
We see $\overline{N}(-m,n)=\overline{N}(m,n)$ and so
$\overline{N}(k,t,n)=\overline{N}(t-k,t,n)$. Similarly we have
$\overline{M}(k,t,n)=\overline{M}(t-k,t,n)$,
$N2(k,t,n)=N2(t-k,t,n)$, and
$M2(k,t,n)=M2(t-k,t,n)$.

We will show the following:
\begin{theorem}\label{SptBarRankCrank}
\begin{align}
	(1-z)(1-z^{-1})\SB(z,q)
	&= 
		\sum_{n=0}^\infty\sum_{m=-\infty}^\infty 
			(\overline{N}(m,n)-\overline{M}(m,n))z^mq^n
.
\end{align}
\end{theorem}

\begin{theorem}\label{MSptRankCrank}
\begin{align}
	(1-z)(1-z^{-1})\STwoB(z,q)
	&= 
		\sum_{n=0}^\infty\sum_{m=-\infty}^\infty 
			(N2(m,n)-M2(m,n))z^mq^n
.
\end{align}
\end{theorem}

\begin{theorem}\label{SptBar2RankCrank}
\begin{align}
	(1-z)(1-z^{-1})\SB_2(z,q)
	=&  
		\sum_{n=0}^\infty\sum_{m=-\infty}^\infty 
			\Parans{\frac{\overline{N}(m,n)}{2}-\overline{M}(m,n)}z^mq^n
		\nonumber\\		
		&
		+
		\frac{\aqprod{-q}{q}{\infty}}{\aqprod{q}{q}{\infty}}  
		\Parans{
			\frac{1}{2}
			+
			\sum_{n=1}^\infty 
				\frac{(1-z)(1-z^{-1})(-1)^nq^n}{(1-zq^n)(1-z^{-1}q^n)}			
		}
.
\end{align}
\end{theorem}

\begin{theorem}\label{SptBar1RankCrank}
\begin{align}
	(1-z)(1-z^{-1})\SB_1(z,q)
	=& 
		\sum_{n=0}^\infty\sum_{m=-\infty}^\infty 
			\frac{\overline{N}(m,n)}{2}z^mq^n
		-
		\frac{\aqprod{-q}{q}{\infty}}{\aqprod{q}{q}{\infty}}  
		\Parans{
			\frac{1}{2}
			+
			\sum_{n=1}^\infty 
				\frac{(1-z)(1-z^{-1})(-1)^nq^n}{(1-zq^n)(1-z^{-1}q^n)}			
		}
.
\end{align}
\end{theorem}

In \cite{LO1} Lovejoy and Osburn determined formulas
for the differences of $\overline{N}(s,\ell,\ell n+d)$ for $\ell=3,5$
and in \cite{LO2} they did the same for $N2(s,\ell,\ell n+d)$.
From these difference formulas, we know the 3-dissection and 5-dissection 
for the generating functions of $\overline{N}(m,n)$ and $N2(m,n)$. 
In particular, we will have the following.

\begin{theorem}\label{ThreeDissectionForDysonRank}
\begin{align}
	\sum_{n=0}^\infty\sum_{m=-\infty}^\infty
	\overline{N}(m,n)\zeta_3^mq^n
	&= \overline{N}_{0,3}(q^3) 
		+ q\overline{N}_{1,3}(q^3) 
		+ q^2\overline{N}_{2,3}(q^3)
\end{align}
where
\begin{align}
	\label{EqN03}
	\overline{N}_{0,3}(q)
	&=
		\frac{\aqprod{q^3}{q^3}{\infty}^4\aqprod{q^2}{q^2}{\infty}}
			{\aqprod{q}{q}{\infty}^2\aqprod{q^6}{q^6}{\infty}^2}
,
	\\
	\label{EqN13}
	\overline{N}_{1,3}(q)
	&=
		2\frac{\aqprod{q^3}{q^3}{\infty}\aqprod{q^6}{q^6}{\infty}}
			{\aqprod{q}{q}{\infty}}
.
\end{align}
\end{theorem}

\begin{theorem}\label{FiveDissectionForDysonRank}
\begin{align}
	\sum_{n=0}^\infty\sum_{m=-\infty}^\infty
	\overline{N}(m,n)\zeta_5^mq^n
	&=
		\overline{N}_{0,5}(q^5) 
		+ q\overline{N}_{1,5}(q^5) 
		+ q^2\overline{N}_{2,5}(q^5)
		+ q^3\overline{N}_{3,5}(q^5) 
		+ q^4\overline{N}_{4,5}(q^5)
\end{align}
where
\begin{align}
	\label{EqN05}
	\overline{N}_{0,5}(q)
	&=
		\frac{\aqprod{q^4,q^6}{q^{10}}{\infty}\aqprod{q^5}{q^5}{\infty}^2}
			{\aqprod{q^2,q^3}{q^5}{\infty}^2\aqprod{q^{10}}{q^{10}}{\infty}}
		+
		2(\zeta_5+\zeta_5^{-1})q
			\frac{\aqprod{q^{10}}{q^{10}}{\infty}}
			{\aqprod{q^3,q^4,q^6,q^7}{q^{10}}{\infty}}
,	\\
	\label{EqN35}
	\overline{N}_{3,5}(q)
	&=
		\frac{2(1-\zeta_5-\zeta_5^{-1})\aqprod{q^{10}}{q^{10}}{\infty}}
			{\aqprod{q^2,q^3}{q^{5}}{\infty}}
.
\end{align}
\end{theorem}

\begin{theorem}\label{ThreeDissectionForM2Rank}
\begin{align}
	\sum_{n=0}^\infty\sum_{m=-\infty}^\infty
	N2(m,n)\zeta_3^mq^n
	&=N2_{0,3}(q^3) 
		+ q N2_{1,3}(q^3) 
		+ q^2 N2_{2,3}(q^3)
\end{align}
where
\begin{align}
	\label{EqN213}
	N2_{1,3}(q)
	&=
		\frac{\aqprod{q^6}{q^6}{\infty}^4}
			{\aqprod{q^2}{q^2}{\infty}\aqprod{q^3}{q^3}{\infty}\aqprod{q^{12}}{q^{12}}{\infty}}
.
\end{align}
\end{theorem}

\begin{theorem}\label{FiveDissectionForM2Rank}
\begin{align}
	\sum_{n=0}^\infty\sum_{m=-\infty}^\infty
	N2(m,n)\zeta_5^mq^n
	&=		N2_{0,5}(q^5) 
		+ q N2_{1,5}(q^5) 
		+ q^2 N2_{2,5}(q^5)
		+ q^3 N2_{3,5}(q^5) 
		+ q^4 N2_{4,5}(q^5)
\end{align}
where
\begin{align}
	\label{EqN215}
	N2_{1,5}(q)
	&=
		\frac{\aqprod{-q^5,q^{10}}{q^{10}}{\infty}}
			{\aqprod{q^2,q^8}{q^{10}}{\infty}}
.	\\
	\label{EqN235}
	N2_{3,5}(q)
	&=
		(\zeta_5+\zeta_5^4)\frac{\aqprod{-q^5,q^{10}}{q^{10}}{\infty}}
			{\aqprod{q^4,q^6}{q^{10}}{\infty}}
.
\end{align}
\end{theorem}

The terms $\overline{N}_{2,3}(q)$, $\overline{N}_{1,5}(q)$, 
$\overline{N}_{2,5}(q)$, $\overline{N}_{4,5}(q)$, $N2_{0,3}(q)$, 
$N2_{2,3}(q)$, $N2_{0,5}(q)$, $N2_{2,5}(q)$, and $N2_{4,5}(q)$ are
also products and series in $q$ and follow from the
difference formulas of Lovejoy and Obsurn in \cite{LO1} and \cite{LO2}. 
However we will not need to use them here.

We will determine dissections for the cranks and other series.
In particular, we will prove the following.

\begin{theorem}\label{ThreeDissectionOverpartionCrank}
\begin{align}
	\sum_{n=0}^\infty\sum_{m=-\infty}^\infty
	\overline{M}(m,n)\zeta_3^mq^n
	&= \overline{M}_{0,3}(q^3) 
		+ q\overline{M}_{1,3}(q^3) 
		+ q^2\overline{M}_{2,3}(q^3)
\end{align}
where
\begin{align}
	\label{EqM03}
	\overline{M}_{0,3}(q)
	&=
		\frac{\aqprod{q^3}{q^3}{\infty}^4\aqprod{q^2}{q^2}{\infty}}
		{\aqprod{q}{q}{\infty}^2\aqprod{q^6}{q^6}{\infty}^2}
,	\\
	\label{EqM13}
	\overline{M}_{1,3}(q)
	&=
		-\frac{\aqprod{q^6}{q^6}{\infty}\aqprod{q^3}{q^3}{\infty}}
			{\aqprod{q}{q}{\infty}}
,	\\
	\label{EqM23}
	\overline{M}_{2,3}(q)
	&=
		-2\frac{\aqprod{q^6}{q^6}{\infty}^4}
			{\aqprod{q^3}{q^3}{\infty}^2\aqprod{q^2}{q^2}{\infty}}
.
\end{align}
\end{theorem}

\begin{theorem}\label{FiveDissectionOverpartionCrank}
\begin{align}
	\sum_{n=0}^\infty\sum_{m=-\infty}^\infty
	\overline{M}(m,n)\zeta_5^mq^n
	&=		\overline{M}_{0,5}(q^5) 
		+ q\overline{M}_{1,5}(q^5) 
		+ q^2\overline{M}_{2,5}(q^5)
		+ q^3\overline{M}_{3,5}(q^5) 
		+ q^4\overline{M}_{4,5}(q^5)
\end{align}
where
\begin{align}
	\label{EqM05}
	\overline{M}_{0,5}(q) 
	&= 
		\frac{\aqprod{q^4,q^6,q^{10}}{q^{10}}{\infty}}
			{\aqprod{q,q^4}{q^5}{\infty}\aqprod{q^2,q^8}{q^{10}}{\infty}} 
		- q(\zeta_5+\zeta_5^4)
		\frac{\aqprod{q^2,q^8,q^{10}}{q^{10}}{\infty}}
				{\aqprod{q^2,q^3}{q^5}{\infty}\aqprod{q^4,q^6}{q^{10}}{\infty}}
,	\\
	\label{EqM15}
	\overline{M}_{1,5}(q) 
	&= 
		(\zeta_5+\zeta_5^4)
		\frac{\aqprod{q^4,q^6,q^{10}}{q^{10}}{\infty}}
				{\aqprod{q^2,q^3}{q^5}{\infty}\aqprod{q^2,q^8}{q^{10}}{\infty}}
,	\\
	\label{EqM25}
	\overline{M}_{2,5}(q) 
	&= 
		-\frac{\aqprod{q^{10}}{q^{10}}{\infty}}{\aqprod{q,q^4}{q^5}{\infty}}
,	\\
	\label{EqM35}
	\overline{M}_{3,5}(q) 
	&= 
		-(\zeta_5+\zeta_5^4)
		\frac{\aqprod{q^{10}}{q^{10}}{\infty}}
			{\aqprod{q^2,q^3}{q^5}{\infty}}
,	\\
	\label{EqM45}
	\overline{M}_{4,5}(q) 
	&= 
		-\frac{\aqprod{q^2,q^8,q^{10}}{q^{10}}{\infty}}
			{\aqprod{q,q^4}{q^5}{\infty}\aqprod{q^4,q^6}{q^{10}}{\infty}}
.
\end{align}
\end{theorem}

\begin{theorem}\label{ThreeDissectionM2Crank}
\begin{align}
	\sum_{n=0}^\infty\sum_{m=-\infty}^\infty
	M2(m,n)\zeta_3^mq^n
	&= \overline{M}_{0,3}(q^3) 
		+ q\overline{M}_{1,3}(q^3) 
		+ q^2\overline{M}_{2,3}(q^3)
\end{align}
where
\begin{align}
	\label{EqM203}
	M2_{0,3}(q)
	&=
		\frac{\aqprod{q^6}{q^6}{\infty}^{10}\aqprod{q^4}{q^4}{\infty}\aqprod{q}{q}{\infty}}
			{\aqprod{q^{12}}{q^{12}}{\infty}^4\aqprod{q^3}{q^3}{\infty}^4\aqprod{q^2}{q^2}{\infty}^3}
,	\\
	\label{EqM213}
	M2_{1,3}(q)
	&=
		\frac{\aqprod{q^6}{q^6}{\infty}^4}
			{\aqprod{q^{12}}{q^{12}}{\infty}\aqprod{q^3}{q^3}{\infty}\aqprod{q^2}{q^2}{\infty}}
,	\\
	\label{EqM223}
	M2_{2,3}(q)
	&=
		-2\frac{\aqprod{q^{12}}{q^{12}}{\infty}^2\aqprod{q^3}{q^3}{\infty}^2\aqprod{q^2}{q^2}{\infty}}
			{\aqprod{q^6}{q^6}{\infty}^2\aqprod{q^4}{q^4}{\infty}\aqprod{q}{q}{\infty}}
		.
\end{align}
\end{theorem}

\begin{theorem}\label{FiveDissectionM2Crank}
\begin{align}
	\sum_{n=0}^\infty\sum_{m=-\infty}^\infty
	M2(m,n)\zeta_5^mq^n
	&= 	M2_{0,5}(q^5) 
		+ q M2_{1,5}(q^5) 
		+ q^2 M2_{2,5}(q^5)
		+ q^3 M2_{3,5}(q^5) 
		+ q^4 M2_{4,5}(q^5)
\end{align}
where
\begin{align}
	\label{EqM205}
	M2_{0,5}(q)
	&=
		\frac{\aqprod{-q^3,-q^5,-q^7,q^{10}}{q^{10}}{\infty}}
			{\aqprod{-q,q^4,q^6,-q^9}{q^{10}}{\infty}} 	
,	\\
	\label{EqM215}
	M2_{1,5}(q)
	&=
		\frac{\aqprod{-q^5,q^{10}}{q^{10}}{\infty}}
			{\aqprod{q^2,q^8}{q^{10}}{\infty}}
,	\\
	\label{EqM225}
	M2_{2,5}(q)
	&=
		(\zeta_5+\zeta_5^4)
		\frac{\aqprod{q^2,-q^3,-q^5,-q^7,q^8,q^{10}}{q^{10}}{\infty}}
			{\aqprod{-q,q^4,q^4,q^6,q^6,-q^9}{q^{10}}{\infty}}
		-\frac{\aqprod{-q,q^4,-q^5,q^6,-q^9,q^{10}}{q^{10}}{\infty}}
			{\aqprod{q^2,q^2,-q^3,-q^7,q^8,q^8}{q^{10}}{\infty}}
,	\\
	\label{EqM235}
	M2_{3,5}(q)
	&=
		(\zeta_5+\zeta_5^4)
		\frac{\aqprod{-q^5,q^{10}}{q^{10}}{\infty}}
			{\aqprod{q^4,q^6}{q^{10}}{\infty}}
,	\\
	\label{EqM245}
	M2_{4,5}(q)
	&=
		-(\zeta_5+\zeta_5^4)
			\frac{\aqprod{-q,-q^5,-q^9,q^{10}}{q^{10}}{\infty}}
				{\aqprod{q^2,-q^3,-q^7,q^8}{q^{10}}{\infty}}
.
\end{align}
\end{theorem}

\begin{theorem}\label{ThreeDissectionExtraSeries}
\begin{align}
	\frac{\aqprod{-q}{q}{\infty}}{\aqprod{q}{q}{\infty}}  
	\Parans{
		\frac{1}{2}
		+
		\sum_{n=1}^\infty 
			\frac{(1-\zeta_3)(1-\zeta_3^{-1})(-1)^nq^n}
			{(1-\zeta_3q^n)(1-\zeta_3^{-1}q^n)}			
	}
	&= A_0(q^3) + qA_1(q^3) + q^2A_2(q^3)
\end{align}
where
\begin{align}
	\label{EqExtra03}
	A_0(q) 
	&=
		\frac{\aqprod{q^3}{q^3}{\infty}^4\aqprod{q^2}{q^2}{\infty}}
			{2\aqprod{q}{q}{\infty}^2\aqprod{q^6}{q^6}{\infty}^2}
,	\\
	\label{EqExtra13}
	A_1(q)
	&=
		-2\frac{\aqprod{q^6}{q^6}{\infty}\aqprod{q^3}{q^3}{\infty}}
			{\aqprod{q}{q}{\infty}}
,	\\
	\label{EqExtra23}
	A_2(q) 
	&=
		2\frac{\aqprod{q^6}{q^6}{\infty}^4}
			{\aqprod{q^3}{q^3}{\infty}^2\aqprod{q^2}{q^2}{\infty}}
.
\end{align}
\end{theorem}

\begin{theorem}\label{FiveDissectionExtraSeries}
\begin{align}
	\frac{\aqprod{-q}{q}{\infty}}{\aqprod{q}{q}{\infty}}
	\Parans{
		\frac{1}{2}
		+
		\sum_{n=1}^\infty \frac{(1-\zeta_5)(1-\zeta_5^{-1})(-1)^nq^{n}}
			{(1-\zeta_5q^n)(1-\zeta_5^{-1}q^n)}
	}
	&= B_0(q^5) + qB_1(q^5) + q^2B_2(q^5) + q^3B_3(q^5) + q^4B_4(q^5)
\end{align}
where
\begin{align}
	\label{EqExtra05}
	B_0(q)
	&=
		\frac{\aqprod{q^5}{q^5}{\infty}^2\aqprod{q^4,q^6}{q^{10}}{\infty}}
		{2\aqprod{q^{10}}{q^{10}}{\infty}\aqprod{q^2,q^3}{q^5}{\infty}^2}
		+
		(\zeta_5+\zeta_5^{-1})
		\frac{q\aqprod{q^{10}}{q^{10}}{\infty}}
		{\aqprod{q^3,q^4,q^6,q^7}{q^{10}}{\infty}}
,	\\
	\label{EqExtra15}
	B_1(q)
	&=
		(\zeta_5+\zeta_5^{-1}-1)\frac{\aqprod{q^4,q^6,q^{10}}{q^{10}}{\infty}}
		{\aqprod{q^2,q^8}{q^{10}}{\infty}^2 \aqprod{q^3,q^7}{q^{10}}{\infty}}
,	\\
	\label{EqExtra25}
	B_2(q)
	&=
		(1-2\zeta_5-2\zeta_5^{-1})
		\frac{\aqprod{q^{10}}{q^{10}}{\infty}}
			{\aqprod{q,q^9}{q^{10}}{\infty}\aqprod{q^4,q^6}{q^{10}}{\infty}}
,	\\
	\label{EqExtra35}
	B_3(q)
	&=
		-\frac{\aqprod{q^{10}}{q^{10}}{\infty}}
			{\aqprod{q^2,q^3}{q^{5}}{\infty}}
,	\\
	\label{EqExtra45}
	B_4(q)
	&=
		(\zeta_5+\zeta_5^{-1})
		\frac{\aqprod{q^2,q^8,q^{10}}{q^{10}}{\infty}}
			{\aqprod{q,q^9}{q^{10}}{\infty}\aqprod{q^4,q^6}{q^{10}}{\infty}^2}
	.
\end{align}
\end{theorem}

With these dissections, we need only match up the appropriate terms for each congruence.
The congruence for $\sptBar{}{3n}$ of Theorem \ref{TheoremSptBarCongruences} follows from
using (\ref{EqN03}) and (\ref{EqM03}) to get
\begin{align}
	\overline{N}_{0,3} - \overline{M}_{0,3} &= 0,
\end{align}
which along with Theorem \ref{SptBarRankCrank} gives that the coefficients
of $q^{3n}$ in $\SB(\zeta_3,q)$ are zero.

The congruences for $\sptBar{1}{3n}$ and $\sptBar{1}{5n}$ follow from using
(\ref{EqN03}), (\ref{EqExtra03}), (\ref{EqN05}), and (\ref{EqExtra05})
to get
\begin{align}
	\frac{\overline{N}_{0,3}}{2} - A_0 &= 0,
	\\
	\frac{\overline{N}_{0,5}}{2} - B_0 &= 0,
\end{align}
and Theorem \ref{SptBar1RankCrank}.

The congruences for $\sptBar{2}{3n}$, $\sptBar{2}{3n+1}$, and $\sptBar{2}{5n+3}$ follow from
using 
(\ref{EqN03}), (\ref{EqExtra03}), (\ref{EqM03}),
(\ref{EqN13}), (\ref{EqExtra13}), (\ref{EqM13}),
(\ref{EqN35}), (\ref{EqExtra35}), and (\ref{EqM35})
to get
\begin{align}
	\frac{\overline{N}_{0,3}}{2} + A_0 - \overline{M}_{0,3} &= 0,
	\\
	\frac{\overline{N}_{1,3}}{2} + A_1 - \overline{M}_{1,3} &= 0,
	\\
	\frac{\overline{N}_{3,5}}{2} + B_3 - \overline{M}_{3,5} &= 0,
\end{align}
and Theorem \ref{SptBar2RankCrank}.

Lastly the congruences for $\Mspt{3n+1}$, $\Mspt{5n+1}$, and $\Mspt{5n+3}$
follow from using
(\ref{EqN213}), (\ref{EqM213}),
(\ref{EqN215}), (\ref{EqM215}),
(\ref{EqN235}), and (\ref{EqM235})
to get
\begin{align}
	N2_{1,3} - M2_{1,3} &= 0,
	\\
	N2_{1,5} - M2_{1,5} &= 0,
	\\
	N2_{3,5} - M2_{3,5} &= 0,	
\end{align}
and applying Theorem \ref{MSptRankCrank}.

For Theorem \ref{TheoremSptBar1CongruenceOdd} we have to do a little better. 
In particular we will prove the following.
\begin{theorem}\label{TheoremTwoDissectionSBars}
\begin{align}
	\label{TwoDissectionSBar}
	\SB(i,q) 
	&=
	\sum_{n=1}^\infty q^{n^2} - \sum_{n=1}^\infty (-1)^nq^{2n^2},
	\\
	\label{TwoDissectionS1Bar}
	\SB_1(i,q)
	&= 
	\sum_{n=1} q^{(2n-1)^2},
	\\
	\label{TwoDissectionS2Bar}
	\SB_2(i,q) 
	&=
	\sum_{n=1}^\infty q^{(2n)^2} - \sum_{n=1}^\infty (-1)^nq^{2n^2}
.
\end{align}
\end{theorem}
Considering just $\sptBar{1}{n}$, we have
\begin{align}
	\sum_{n=1} q^{(2n-1)^2}
	&=
	\sum_{n=0}^\infty\sum_{m=-\infty}^\infty 
		N_{\SB_1}(m,n)i^mq^n
	\\
	&=
	\sum_{n=0}^\infty
	\Parans{ 
		N_{\SB_1}(0,4,n)  -N_{\SB_1}(2,4,n)
		+i(	N_{\SB_1}(1,4,n)- N_{\SB_1}(3,4,n))
	}q^n
	\\
	&=	
	\sum_{n=0}^\infty
	\Parans{ 
		N_{\SB_1}(0,4,n)  -N_{\SB_1}(2,4,n)
	}q^n.
\end{align}
But
\begin{align}
	\sptBar{1}{n} &=
		N_{\SB_1}(0,4,n) + N_{\SB_1}(1,4,n)
		+N_{\SB_1}(2,4,n) + N_{\SB_1}(3,4,n)
	\\
	&= N_{\SB_1}(0,4,n) + 2N_{\SB_1}(1,4,n)
		+N_{\SB_1}(2,4,n)
\end{align}
and so we see
\begin{align}
	\sptBar{1}{n} &\equiv 1 \pmod{2}
\end{align} 
if and only if $n$ is an odd square. The parity of $\sptBar{}{n}$ and 
$\sptBar{2}{n}$ follow in the same fashion.

In \cite{AGL} Andrews, the first author, and Liang also showed that
$N_{\mbox{\rm S}}(m,n)$, the coefficients in S$(z,q)$, to be nonnegative.
The same phenomenon occurs here.

\begin{theorem}\label{TheoremNonNegative}
For all $m$ and $n$ we have
$N_{\SB}(m,n)$,
$N_{\SB_1}(m,n)$,
and
$N_{\SB_2}(m,n)$
are nonnegative.
\end{theorem}

In section $3$ we give combinatorial interpretations of the series 
$\SB(z,q)$, $\SB_1(z,q)$,  $\SB_2(z,q)$, and $\STwoB(z,q)$ in terms of 
weighted vector partitions and then prove Theorem \ref{TheoremNonNegative}. For 
$\SB(z,q)$, $\SB_1(z,q)$ and $\SB_2(z,q)$ 
we define the spt-crank in terms of marked overpartitions, see equation
(\ref{eq:sptcrankbdef}). In Theorem \ref{thm:mainthm} we give a combinatorial
interpretation of each of the spt-overpartition congruences in Theorem 
\ref{TheoremSptBarCongruences} in terms of marked overpartitions. 
In section $4$ we prove the theorems on expressing $\SB(z,q)$, $\SB_1(z,q)$, 
$\SB_2(z,q)$, and $\STwoB(z,q)$
in terms of the difference between a rank and crank. In section $5$
we prove the various dissections. In section 6 we conclude with remarks on the
nonnegativity of the coefficients of $\STwoB(z,q)$; a recent result
by Andrews, Chan, Kim, and Osburn \cite{ACKO} on the first moments 
for the rank and crank of overpartitions; and the remaining spt function of
\cite{BLO1}.

\section{Combinatorial Interpretations} \label{sec:combint}

In this section we provide combinatorial interpretations of the coefficients
in the series $\SB(z,q)$, $\SB_1(z,q)$, 
$\SB_2(z,q)$, and $\STwoB(z,q)$. For all four series we provide
an interpretation in terms of certain vector partitions with four components.
For the three series $\SB(z,q)$, $\SB_1(z,q)$, and 
$\SB_2(z,q)$ we give two additional interpretations --
one in terms of pairs of partitions and finally an interpretations
in terms of marked overpartitions. This final interpretation will give
interpretations of the congruences for overpartitions directly in terms of
the overpartitions themselves.

\subsection{Vector partitions and $\SB$-partitions} \label{subsec:vp}

The coefficients in the series $\SB(z,q)$, $\SB_1(z,q)$, 
$\SB_2(z,q)$, and $\STwoB(z,q)$ can be interpreted in terms of cranks 
of vector partitions. This can be done with vectors with 4 components, 
each a partition with certain restrictions.

We let $\overline{V} = \mathcal{D}\times\mathcal{P}\times\mathcal{P}\times\mathcal{D}$, 
where $\mathcal{P}$ denotes the set of all partitions and $\mathcal{D}$ 
denotes the set of all partitions into distinct parts. For a partition 
$\pi$ we let $s(\pi)$ denote the smallest part of $\pi$ 
(with the convention that the empty partition has smallest part $\infty$),
$\#(\pi)$ the number of parts in $\pi$,
and $|\pi|$ the sum of the parts of $\pi$. 
For $\vec{\pi} = (\pi_1,\pi_2,\pi_3,\pi_4) \in\overline{V}$, we define the weight
$\omega(\vec{\pi}) = (-1)^{\#(\pi_1)-1 }$, the 
$\mbox{crank}(\vec{\pi}) = \#(\pi_2) - \#(\pi_3)$,
and the norm $|\vec{\pi}|=|\pi_1|+|\pi_2|+|\pi_3|+|\pi_4|$. We say $\vec{\pi}$ is 
a vector partition of $n$ if $|\vec{\pi}|=n$.
 
We then let $\SB$ denote the subset of $\overline{V}$ given by
\begin{align}
	\SB = 
	\left\{
		(\pi_1,\pi_2,\pi_3,\pi_4)\in \overline{V}
		:
		1\le s(\pi_1)<\infty, s(\pi_1) \le s(\pi_2), s(\pi_1) \le s(\pi_3), s(\pi_1) < s(\pi_4)
	\right\}.
\end{align}

We let $\SB_1$ and $\SB_2$ denote the subsets of $\SB$ with
$s(\pi_1)$ odd and even, respectively.

We see then that the number of vector partitions of $n$ in $\SB$ 
with crank $m$ counted according to the weight $\omega$ is 
exactly $N_{\SB}(m,n)$. Similarly the number of vector partitions 
of $n$ in $\SB_1$ with crank $m$ counted according to 
the weight $\omega$ is $N_{\SB_1}(m,n)$, and the number of vector partitions 
of $n$ in in $\SB_2(m,n)$ with crank $m$ counted according to the weight $\omega$ 
is $N_{\SB_2}(m,n)$.

We let $n_o(\pi)$ and $n_e(\pi)$ denote the number of odd and even parts, respectively, of $\pi$.
We let $\STwoB$ denote the subset of $\SB$ given by
\begin{align}
	\STwoB
	&=
	\left\{
		(\pi_1,\pi_2,\pi_3,\pi_4)\in \SB
		:
		n_o(\pi_1) = 0,
		n_o(\pi_2) = 0,
      n_o(\pi_3) = 0,		
		n_e(\pi_4) = 0 
	\right\}.
\end{align}
Then $N_{\STwoB}(m,n)$ is the number of vector partitions from $\STwoB$ of $n$ with crank $m$ counted
according to the weight $\omega$.

For each of the four spt functions, we give an example to illustrate a congruence.

\begin{example}[n=3] The $4$ overpartitions $3$ of with smallest part 
not overlined are $3$, $2+1$, $\overline{2}+1$, and $1+1+1$. We have then
$\sptBar{}{3}=6$. There are $8$ vector partitions from $\SB$ of $3$.
These vector partitions along  with their weights and cranks
are given as follows. 
$$
	\begin{array}{cccc}
		\mbox{$\SB$-vector partition } & \mbox{weight} & \mbox{crank} & \mbox{(mod 3)}
		\\
		\mbox{[1, --, --, 2]}	&1		&	0	& 0
		\\
		\mbox{[1, --, 1+1, --]}	&1		& -2	& 1
		\\
		\mbox{[1, --, 2, --]}	&1		& -1	& 2
		\\
		\mbox{[1, 1, 1, --]}	&1		& 0	& 0
		\\
		\mbox{[1, 1+1, --, --]}	&1		& 2	& 2
		\\
		\mbox{[1, 2, --, --]}	&1		& 1	&1
		\\
		\mbox{[1+2, --, --]}	&-1	& 0	&0
		\\
		\mbox{[3, --, --, --]}	&	1	& 0	&0
	\end{array}
$$
Here we have used -- to indicate the empty partition.  
We see that
\begin{align}
	N_{\SB}(0,3,3)&=N_{\SB}(1,3,3)=N_{\SB}(2,3,3)
	=2
	= \frac{1}{3}\sptBar{}{3}.
\end{align}

We note that the overpartitions of $3$ all have smallest part odd, so that
$\sptBar{1}{3}=6$. Also all vector partitions from $\SB$ of $3$ are also 
from $\SB_1$. Thus we also have 
\begin{align}
	N_{\SB_1}(0,3,3)&=N_{\SB_1}(1,3,3)=N_{\SB_1}(2,3,3)=2
	= \frac{1}{3}\sptBar{1}{3}.
\end{align}
\end{example}

\begin{example}[n=4] The $2$ overpartitions of $4$ with smallest part even 
and not overlined are
$4$ and $2+2$, so $\sptBar{2}{4}=3$. There are $3$ vector partitions from 
$\SB_2$ of $4$. These vector partitions are as follows.
$$
	\begin{array}{cccc}
		\mbox{$\SB_2$-vector partition } & \mbox{weight} & \mbox{crank} & \mbox{(mod 3)}
		\\
		\mbox{[2, 2, --, --]}	&1		& 1 & 1
		\\
		\mbox{[2, --, 2, --]}	&1		& -1 & 2
		\\
		\mbox{[4, --, --, --]}	&1		& 0 & 0
	\end{array}
$$
We see that
\begin{align}
	N_{\SB_2}(0,3,4)&=N_{\SB_2}(1,3,4)=N_{\SB_2}(2,3,4)=1
	= \frac{1}{3}\sptBar{2}{4}.
\end{align}
\end{example}

\begin{example}[n=6]
The $3$ partitions of $6$ without repeated odd parts and smallest part even are
$6$, $4+2$, $2+2+2$, so that $\Mspt{6}=5$. There are $7$ vector partitions of $6$
from $\STwoB$. These vector partitions are as follows.
$$
	\begin{array}{cccc}
		\mbox{$\STwoB$-vector partition }	& \mbox{weight} & \mbox{crank} & \mbox{(mod 5)}
		\\
		\mbox{[2, --, 4, --]}		&1		& -1 & 4
		\\
		\mbox{[2, --, 2+2, --]}		&1		& -2 & 3
		\\
		\mbox{[2, 2, 2, --]}		&1		& 0 & 0
		\\
		\mbox{[2, 4, --, --]}		&1		& 1 & 1
		\\
		\mbox{[2, 2+2, --, --]}		&1		& 2 & 2
		\\
		\mbox{[4+2, --, --, --]}		&-1	& 0 & 0
		\\
		\mbox{[6, --, --, --]}		&1		& 0 & 0
	\end{array}
$$
We see that
\begin{align}
	N_{\STwoB}(0,5,6)&=N_{\STwoB}(1,5,6)=N_{\STwoB}(2,5,6)
	=N_{\STwoB}(3,5,6)=N_{\STwoB}(4,5,6) = 1
	= \frac{1}{5}\Mspt{6}.
\end{align}
\end{example}


\subsection{$\SPB$-partition pairs} \label{subsec:spp}

In this section we prove that
\beq
N_{\SB}(m, n) \ge 0,
\mylabel{eq:NSBineq}
\eeq
for all $m$, $n$ and provide a combinatorial interpretation in terms of
partition pairs. 

\subsubsection{Proof of nonnegativity} \label{subsubsec:nonneg}

\begin{align}
{\SB}(z,q) &= \sum_{n=1}^\infty \sum_m N_{\SB}(m,n) z^m q^n 
\mylabel{eq:nonnegeq}\\
         &= \sum_{n=1}^\infty \frac{q^n (q^{n+1};q)_\infty}
                            {(zq^n;q)_\infty (z^{-1}q^n;q)_\infty}
                        (-q^{n+1};q)_\infty \nonumber\\
         &= \sum_{n=1}^\infty \frac{q^n (q^{2n};q)_\infty}
                            {(zq^n;q)_\infty (z^{-1}q^n;q)_\infty}
                        \frac{(q^{2n+2};q^2)_\infty}
                             {(q^{2n};q)_\infty} \nonumber\\
&= \sum_{n=1}^\infty q^n
   \sum_{k=0}^\infty \frac{ (z^{-1}q^n)^k }
                          {(zq^{n+k};q)_\infty (q)_k}
   \times \frac{1}{(1-q^{2n})} \frac{1}{(q^{2n+1};q^2)_\infty}
\nonumber
\end{align}
by \cite[Prop. 4.1]{BG}. The inequality \eqn{NSBineq} clearly follows.
Replacing $n$ by $2n+1$ and $2n$ in the second line of (\ref{eq:nonnegeq}) 
gives $N_{\SB_1}(m,n)\ge 0$ and $N_{\SB_2}(m,n)\ge 0$, respectively.

\subsubsection{The $\sptcrank$ in terms of partition pairs} 
\label{subsubsec:sptcpp}

We define
\begin{align}
\SPB = \{ \vec{\lambda} = (\lambda_1,\lambda_2)\in\mathcal{P}\times\mathcal{P}
\, &: \, 0 < s(\lambda_1) \le s(\lambda_2) \mylabel{eq:SPBdef}\\
   & \mbox{and all parts of $\lambda_2$ that are $\ge 2s(\lambda_1)+1$
           are odd}\}.
\nonumber
\end{align}
First we show that
\beq
\sptb(n) = \sum_{\substack{ \vec{\lambda}\in{\SPB} \\
                 \abs{\vec{\lambda}}=\abs{\lambda_1}+\abs{\lambda_2}=n}} 1.
\mylabel{eq:sptbid}
\eeq
\begin{align}
&\sum_{n=1}^\infty \sptb(n) q^n = 
 \sum_{n=1}^\infty \frac{q^n (-q^{n+1};q)_\infty}{(1-q^n)^2 (q^{n+1};q)_\infty}
 \mylabel{eq:sptbid1}\\
&= \sum_{n=1}^\infty \frac{q^n}{(1-q^n)^2 (q^{n+1};q)_\infty}
                     \times \frac{(q^{2n+2};q^2)_\infty}{(q^{n+1};q)_\infty}
\nonumber\\
&= \sum_{n=1}^\infty \frac{q^n}{(q^{n};q)_\infty} \times
\frac{1}{(1-q^n)(1-q^{n+1}) \cdots (1-q^{2n}) (q^{2n+1};q^2)_\infty}
\nonumber\\
& = \sum_{n=1}^\infty 
\sum_{\substack{\lambda_1\in\mathcal{P} \\ s(\lambda_1) = n}} 
     q^{\abs{\lambda_1}}
\sum_{\substack{\lambda_2\in\mathcal{P} \\ s(\lambda_2) \ge n\\
     \mbox{\scriptsize all parts in $\lambda_2\ge2n+1$ are odd}}} 
     q^{\abs{\lambda_2}}
\nonumber\\
& = \sum_{n=1}^\infty 
\sum_{\substack{\vec{\lambda}=(\lambda_1,\lambda_2)\in\SPB \\
               s(\lambda_1) = n}} 
     q^{\abs{\vec{\lambda}}}
\nonumber\\
& = \sum_{\vec{\lambda}\in\SPB} q^{\abs{\vec{\lambda}}},
\end{align}
and \eqn{sptbid} follows.

We let $\SPB_1$ be the set of $\vec{\lambda}=(\lambda_1,\lambda_2)\in\SPB$
with $s(\lambda_1)$ odd and
let $\SPB_2$ be the set of $\vec{\lambda}=(\lambda_1,\lambda_2)\in\SPB$
with $s(\lambda_1)$ even. Then in the same fashion we have
\beq
\sptb_1(n) = \sum_{\substack{ \vec{\lambda}\in{\SPB_1} \\
                 \abs{\vec{\lambda}}=n}} 1,
\mylabel{eq:sptb1id}
\eeq
and
\beq
\sptb_2(n) = \sum_{\substack{ \vec{\lambda}\in{\SPB_2} \\
                 \abs{\vec{\lambda}}=n}} 1.
\mylabel{eq:sptb2id}
\eeq

Next we define a $\crankb$ of partition pairs 
$\vec{\lambda}=(\lambda_1,\lambda_2)\in\SPB$ by interpreting the coefficient
of $z^mq^n$ in \eqn{nonnegeq}.
For 
$\vec{\lambda}=(\lambda_1,\lambda_2)\in\SPB$ we define 
\beq
k(\vec{\lambda}) = \mbox{$\#$ of parts $j$ in $\lambda_2$ such that
$s(\lambda_1) \le j \le 2s(\lambda_1)-1$},
\mylabel{eq:kdef}
\eeq
and define
\beq
\crankb(\vec{\lambda}) = 
\begin{cases}
\mbox{($\#$ of parts of $\lambda_1\ge s(\lambda_1)+k) - k$} & \mbox{if $k>0$}\\
\mbox{($\#$ of parts of $\lambda_1) - 1$} & \mbox{if $k=0$,}
\end{cases}
\mylabel{eq:crankbdef}
\eeq
where $k=k(\vec{\lambda})$.  We have
\begin{theorem}
\mylabel{thm:SPthm}
\begin{align}
N_{\SB}(m,n) &= \mbox{$\#$ of $\vec{\lambda}=(\lambda_1,\lambda_2)\in\SPB$
                      with 
                 $\abs{\vec{\lambda}}=n$ and $\crankb(\vec{\lambda})=m$},
\mylabel{eq:SPthmi}\\
N_{\SB_1}(m,n) &= \mbox{$\#$ of $\vec{\lambda}=(\lambda_1,\lambda_2)\in\SPB_1$
                      with 
                 $\abs{\vec{\lambda}}=n$ and $\crankb(\vec{\lambda})=m$},
\mylabel{eq:SPthmii}\\
N_{\SB_2}(m,n) &= \mbox{$\#$ of $\vec{\lambda}=(\lambda_1,\lambda_2)\in\SPB_2$
                      with 
                 $\abs{\vec{\lambda}}=n$ and $\crankb(\vec{\lambda})=m$}.
\mylabel{eq:SPthmiii}
\end{align}
\end{theorem}
\begin{proof}
From \eqn{nonnegeq} we have 
\begin{align}
{\SB}(z,q) &= \sum_{n=1}^\infty \sum_m N_{\SB}(m,n) z^m q^n 
\mylabel{eq:nonnegeqAGAIN}\\
&= \sum_{n=1}^\infty 
   \sum_{k=0}^\infty 
   \frac{q^n}{(1-q^n) \cdots (1-q^{n+k-1}) (zq^{n+k};q)_\infty}
   \times z^{-k} q^{nk} \qbin{n+k-1}{k} 
   \times \frac{1}{(1-q^{2n})} \frac{1}{(q^{2n+1};q^2)_\infty}
\nonumber\\
&= \sum_{n=1}^\infty 
   \frac{q^n}{(zq^{n};q)_\infty}
   \times \frac{1}{(1-q^{2n})} \frac{1}{(q^{2n+1};q^2)_\infty}
\nonumber\\
&\quad +  \sum_{n=1}^\infty 
   \sum_{k=1}^\infty 
   \frac{q^n}{(1-q^n) \cdots (1-q^{n+k-1}) (zq^{n+k};q)_\infty}
   \times z^{-k} q^{nk} \qbin{n+k-1}{k} 
   \times \frac{1}{(1-q^{2n})} \frac{1}{(q^{2n+1};q^2)_\infty}
\nonumber
.
\end{align}
We note that the $q$-binomial coefficient $\qbin{n+k-1}{k}$ is the
generating function for partitions into parts $\le n-1$ with
number of parts $\le k$. Thus we see that 
$q^{nk} \qbin{n+k-1}{k}$ is the generating function for partitions into
exactly $k$ parts $j$, where $n \le j \le 2n-1$. Hence
\begin{align}
{\SB}(z,q) &= 
\sum_{n=1}^\infty 
\sum_{\substack{\vec{\lambda}=(\lambda_1,\lambda_2)\in\SPB \\
               s(\lambda_1) = n,\ k(\vec{\lambda})=0}} 
      z^{\#(\lambda_1)-1} q^{\abs{\vec{\lambda}}}
\mylabel{eq:SBid2}
+ \sum_{n=1}^\infty 
\sum_{k=1}^\infty
\sum_{\substack{\vec{\lambda}=(\lambda_1,\lambda_2)\in\SPB \\
               s(\lambda_1) = n,\ k(\vec{\lambda})=k}} 
      z^{\crankb(\vec{\lambda})} q^{\abs{\vec{\lambda}}}
\\
&= 
\sum_{\vec{\lambda}\in\SPB} 
      z^{\crankb(\vec{\lambda})} q^{\abs{\vec{\lambda}}}.
\end{align}
The result \eqn{SPthmi} follows. The results
\eqn{SPthmii}, \eqn{SPthmiii}, follow in a similar fashion.
\end{proof}

\subsubsection{Examples} 
\label{subsubsec:egs}
We illustrate our combinatorial interpretation of each
$\sptb$, $\sptb_2$, $\sptb_2$ congruence in terms of the $\crankb$
of $\SPB$-partition pairs.

\begin{example}[$n=3$]
The overpartitions of $3$ with smallest parts not overlined are
$3$, $2+1$, $\overline{2}+1$, $1+1+1$ so that $\sptb(3)=6$.
There are $6$ $\SPB$-partition pairs of $3$.
$$
\begin{array}{cccc}
\mbox{$\SPB$-partition pair} &k   &\crankb  & \pmod{3} \\
\mbox{[3, --]}               &0     &0          &0           \\
\mbox{[2+1, --]}               &0     &1          &1           \\
\mbox{[1+1+1, --]}               &0     &2          &2           \\
\mbox{[1+1, 1]}               &1     &-1          &2           \\
\mbox{[1, 1+1]}               &2     &-2          &1           \\
\mbox{[1, 2]}               &0     &0          &0           \\
\end{array}
$$
We see that
$$
N_{\SB}(0,3,3) = N_{\SB}(1,3,3) = N_{\SB}(2,3,3) = 2 = \frac{\sptb(3)}{3}.
$$
\end{example}

\begin{example}[$n=5$]
There are $10$ overpartitions of $5$ with smallest parts odd
and not overlined:
$$
\begin{array}{lllll}
5,      &4+1,  &\overline{4}+1,    &3+1+1,     &\overline{3}+1+1 \\
2+2+1,  &\overline{2}+2+1,  &2+1+1+1,  &\overline{2}+1+1+1, &1+1+1+1+1,
\end{array}
$$
so that $\sptb_1(5)=20$. There are 20 $\SPB_1$-partitions pairs of $5$.

$$
\begin{array}{cccc}
\mbox{$\SPB_1$-partition pair} &k   &\crankb  & \pmod{5} \\
\mbox{[1, 1+1+1+1]} & 4 & -4 & 1 \\
\mbox{[1, 2+1+1]} & 2 & -2 & 3 \\
\mbox{[1, 2+2]} & 0 & 0 & 0 \\
\mbox{[1, 3+1]} & 1 & -1 & 4 \\
\mbox{[1+1, 1+1+1]} & 3 & -3 & 2 \\
\mbox{[1+1, 2+1]} & 1 & -1 & 4 \\
\mbox{[1+1, 3]} & 0 & 1 & 1 \\
\mbox{[1+1+1, 1+1]} & 2 & -2 & 3 \\
\mbox{[1+1+1, 2]} & 0 & 2 & 2 \\
\mbox{[2+1, 1+1]} & 2 & -2 & 3 \\
\mbox{[2+1, 2]} & 0 & 1 & 1 \\
\mbox{[1+1+1+1, 1]} & 1 & -1 & 4 \\
\mbox{[2+1+1, 1]} & 1 & 0 & 0 \\
\mbox{[3+1, 1]} & 1 & 0 & 0 \\
\mbox{[1+1+1+1+1, --]} & 0 & 4 & 4 \\
\mbox{[2+1+1+1, --]} & 0 & 3 & 3 \\
\mbox{[2+2+1, --]} & 0 & 2 & 2 \\
\mbox{[3+1+1, --]} & 0 & 2 & 2 \\
\mbox{[4+1, --]} & 0 & 1 & 1 \\
\mbox{[5, --]} & 0 & 0 & 0 
\end{array}
$$
We see that
$$
N_{\SB_1}(0,5,5) = N_{\SB_1}(1,5,5) = N_{\SB_1}(2,5,5) 
=N_{\SB_1}(3,5,5) = N_{\SB_1}(4,5,5) = 4 = \frac{\sptb_1(5)}{5}.
$$
\end{example}

\begin{example}[$n=8$]
There are $9$ overpartitions of $8$ with smallest parts even
and not overlined:
$$
\begin{array}{lllll}
8,      &6+2,  &\overline{6}+2,    &3+3+2,     &\overline{3}+3+2 \\
4+4,  &4+2+2,  &\overline{4}+2+2,  &2+2+2+2, &                  
\end{array}
$$
so that $\sptb_2(8)=15$. There are 15 $\SPB_2$-partitions pairs of $8$.

$$
\begin{array}{cccc}
\mbox{$\SPB_2$-partition pair} &k   &\crankb  & \pmod{5} \\
\mbox{[2, 2+2+2]} & 3 & -3 & 2 \\
\mbox{[2, 3+3]} & 2 & -2 & 3 \\
\mbox{[2, 4+2]} & 1 & -1 & 4 \\
\mbox{[2+2, 2+2]} & 2 & -2 & 3 \\
\mbox{[2+2, 4]} & 0 & 1 & 1 \\
\mbox{[4, 4]} & 1 & -1 & 4 \\
\mbox{[3+2, 3]} & 1 & 0 & 0 \\
\mbox{[2+2+2, 2]} & 1 & -1 & 4 \\
\mbox{[4+2, 2]} & 1 & 0 & 0 \\
\mbox{[2+2+2+2, --]} & 0 & 3 & 3 \\
\mbox{[3+3+2, --]} & 0 & 2 & 2 \\
\mbox{[4+2+2, --]} & 0 & 2 & 2 \\
\mbox{[4+4, --]} & 0 & 1 & 1 \\
\mbox{[6+2, --]} & 0 & 1 & 1 \\
\mbox{[8, --]} & 0 & 0 & 0 
\end{array}
$$
We see that
$$
N_{\SB_2}(0,5,8) = N_{\SB_2}(1,5,8) = N_{\SB_2}(2,5,8) 
=N_{\SB_2}(3,5,8) = N_{\SB_2}(4,5,8) = 3 = \frac{\sptb_2(8)}{5}.
$$
\end{example}

\subsection{SPT-crank for marked overpartitions} \label{subsec:sptcmop}

Andrews, Dyson and Rhoades \cite{An-Dy-Rh}
defined a \textit{marked partition} as a pair $(\lambda,k)$
where $\lambda$ is a partition and $k$ is an integer identifying
one of its smallest parts; i.e.\ $k=1$, $2$, \dots, $\nu(\lambda)$,
where $\nu(\lambda)$ is the number of smallest parts of $\lambda$.
They asked for a statistic like the crank which would divide
the relevant marked partitions into $t$ equal classes for $t=5$, $7$, $13$
thus explaining the congruences 
(\ref{TheoremSptCongruence5np4}),
(\ref{TheoremSptCongruence7np5}),
(\ref{TheoremSptCongruence13np6}).
This problem was solved
by Chan, Ji and Zang \cite{Ch-Ji-Za} for the cases $t=5$, $7$.
They defined an spt-crank for double marked partitions and found a bijection
between double marked partitions and marked partitions. It is an open
problem to define the spt-crank directly in terms of marked partitions.
In this section we solve the analogous problem for overpartitions by 
defining a statistic on marked overpartitions.

\subsubsection{Definition of $\sptcrank$ for marked overpartitions} 
\label{subsubsec:sptcbdef}

We define a \textit{marked overpartition} of $n$ as a pair
$(\pi,j)$ where $\pi$ is an overpartition of $n$ in which the smallest
part is not overlined and $j$ is an integer $1 \le j \le \nu(\pi)$,
where as above $\nu(\pi)$ is the number of smallest parts of $\pi$. It is
clear that
\beq
\sptb(n) = \mbox{$\#$ of marked overpartitions $(\pi,j)$ of $n$.}
\mylabel{eq:sptbeq}
\eeq
For example, there are $6$ marked overpartitions of $3$:
$$
\begin{array}{lllll}
(\overline{2}+1,1)      &(2+1,1),      &(3,1), \\
(1+1+1,1)      &(1+1+1,2),      &(1+1+1,3), \\
\end{array}
$$
so that $\sptb(3)=6$. 

To define the $\sptcrank$ of a marked overpartition we first need
to define a function $k(m,n)$. For a positive integer $m$ we write
$$
m = b(m) \, 2^{j(m)},
$$
where $b(m)$ is odd and $j(m)\ge0$. For integers $m,n$ such that $m\ge n+1$, 
we define $j_0(m,n)$ to be the smallest nonnegative integer
$j_0$ such that
$$
b(m)2^{j_0} \ge n+1.
$$
We define
\beq
k(m,n) = 
\begin{cases}
0,             & \mbox{if $b(m) \ge 2n$,}\\
2^{j(m)-j_0(m,n)},     & \mbox{if $b(m)\,2^{j_0(m,n)} < 2n$,}\\
0,             & \mbox{if $b(m)2^{j_0(m,n)} = 2n$.}
\end{cases}
\mylabel{eq:kdefint}
\eeq
We note that if $j_0(m,n)\ge 1$ then $b(m)2^{j_0(m,n)} \le  2n$ so that the
function $k(m,n)$ is well-defined. For a partition
$$
\pi\,:\, m_1 + m_2 + \cdots + m_a
$$
into distinct parts
$$
m_1 > m_2 > \cdots > m_a \ge n+1
$$
we define the function
\beq
k(\pi,n) = \sum_{j=1}^a k(m_j,n) = \sum_{m\in\pi} k(m,n).
\mylabel{eq:kdefdptn}
\eeq
For a marked overpartition $(\pi,j)$ we let 
$\pi_1$ be the partition formed by the non-overlined parts of $\pi$, 
$\pi_2$ be the partition (into distinct parts) formed by the overlined parts 
of $\pi$, so that
$$
s(\pi_2) > s(\pi_1).
$$
We define a function
\beq
\overline{k}(\pi,j) = \nu(\pi_1) - j + k(\pi_2,s(\pi_1)).
\mylabel{eq:kbdef}
\eeq
Finally we can now define
\beq
\sptcrank(\pi,j) =                 
\begin{cases}
\mbox{($\#$ of parts of $\pi_1\ge s(\pi_1)+\kb) - \kb$}, & 
\mbox{if $\kb=\kb(\pi,j)>0$},\\
\mbox{($\#$ of parts of $\pi_1) - 1$} & 
\mbox{if $\kb=\kb(\pi,j)=0$.}
\end{cases}
\mylabel{eq:sptcrankbdef}
\eeq
We state our main theorem.
\begin{theorem}
\mylabel{thm:mainthm}
\begin{enumerate}
\item[(i)]
The residue of the $\sptcrank$ mod $3$ divides the marked overpartitions
of $3n$ into $3$ equal classes.
\item[(ii)]
The residue of the $\sptcrank$ mod $3$ divides the marked overpartitions
of $3n$ and of $3n+1$ with smallest part even into $3$ equal classes.
\item[(iii)]
The residue of the $\sptcrank$ mod $5$ divides the marked overpartitions
of $5n+3$ with smallest part even into $5$ equal classes.
\item[(iv)]
The residue of the $\sptcrank$ mod $3$ divides the marked overpartitions
of $3n$ with smallest part odd into $3$ equal classes.
\item[(v)]
The residue of the $\sptcrank$ mod $5$ divides the marked overpartitions
of $5n$ with smallest part odd into $5$ equal classes.
\item[(vi)]
The residue of the $\sptcrank$ mod $4$ divides the marked overpartitions
of $n$ into $4$ classes with $2$ classes of equal size and the remaining $2$ classes
are of equal size unless $n$ is a square or twice a square in which case the remaining
$2$ classes differ in size by exactly $1$.
\item[(vii)]
The residue of the $\sptcrank$ mod $4$ divides the marked overpartitions
of $n$ with smallest part odd into $4$ classes with $2$ classes of equal size and 
the remaining $2$ classes
are of equal size unless $n$ is an odd square in which case the remaining
$2$ classes differ in size by exactly $1$.
\item[(viii)]
The residue of the $\sptcrank$ mod $4$ divides the marked overpartitions
of $n$ with smallest part even into $4$ classes with $2$ classes of equal size and 
the remaining $2$ classes
are of equal size unless $n$ is an even square or twice a square in which case the remaining
$2$ classes differ in size by exactly $1$.
\end{enumerate}
\end{theorem}

\subsubsection{Examples} 
\label{subsubsec:sptcbegs}

\begin{example}[$n=3$]
There are $6$ marked overpartitions of $3$ so that $\sptb(3)=6$.
$$
\begin{array}{cccccccc}
(\pi,j) &\pi_1 &\pi_2 &\nu(\pi_1) &k(\pi_2,s(\pi_1))&\kb &\sptcrank &\pmod{3}\\
(\ov{2}+1,1)   & 1 & 2 & 1 & 0 & 0 & 0 & 0\\
(1+1+1,1) & 1+1+1 & \dd & 3 & 0 & 2 & -2 & 1\\
(1+1+1,2) & 1+1+1 & \dd & 3 & 0 & 1 & -1 & 2\\
(1+1+1,3) & 1+1+1 & \dd & 3 & 0 & 0 & 2 & 2\\
(2+1,1) & 2+1 & \dd & 1 & 0 & 0 & 1 & 1\\
(3,1)  & 3 & \dd & 1 & 0 & 0 & 0 & 0
\end{array}
$$
We see that the residue of the $\sptcrank\pmod{3}$ divides the marked 
overpartitions of $3$ into $3$ equal classes. This illustrates
Theorem \thm{mainthm}(i).
\end{example}

\begin{example}[$n=5$]
There are $15$ marked overpartitions of $8$ with smallest part even so that $\sptb_2(8)=15$.
$$
\begin{array}{ccccccccc}
(\pi,j) &\pi_1 &\pi_2 &\nu(\pi_1) &k(\pi_2,s(\pi_1))&\kb &\sptcrank & \pmod{4} &\pmod{5}\\
(\ov{6}+2,1) & 2  & 6  & 1 & 2  & 2  & -2  & 2  &3\\
(\ov{4}+2+2,1) & 2+2  & 4  & 2 & 0  & 1  & -1 &3  &4\\
(\ov{4}+2+2,2) & 2+2  & 4  & 2 & 0  & 0  & 1 & 1  &1\\
(\ov{3}+3+2,1) & 3+2  & 3  & 1 & 1  & 1  & 0 &0  &0\\
(2+2+2+2,1) & 2+2+2+2  & \dd  & 4 & 0  & 3  & -3 &1  &2\\
(2+2+2+2,2) & 2+2+2+2  & \dd  & 4 & 0  & 2  & -2 &2  &3\\
(2+2+2+2,3) & 2+2+2+2  & \dd  & 4 & 0  & 1  & -1 &3  &4\\
(2+2+2+2,4) & 2+2+2+2  & \dd  & 4 & 0  & 0  & 3 &3  &3\\
(3+3+2,1) & 3+3+2  & \dd  & 1 & 0  & 0  & 2 &2  &2\\
(4+2+2,1) & 4+2+2  & \dd  & 2 & 0  & 1  & 0 &0  &0\\
(4+2+2,2) & 4+2+2  & \dd  & 2 & 0  & 0  & 2 &2  &2\\
(6+2,1) & 6+2  & \dd  & 1 & 0  & 0  & 1  &1 &1\\
(4+4,1) & 4+4  & \dd  & 2 & 0  & 1  & -1 &3  &4\\
(4+4,2) & 4+4  & \dd  & 2 & 0  & 0  & 1 &1  &1\\
(8,1) & 8  & \dd  & 1 & 0  & 0  & 0 &0  &0\\
\end{array}
$$
We see that the residue of the $\sptcrank\pmod{5}$ divides the marked 
overpartitions of $8$ with even smallest part into $5$ equal classes. This illustrates
Theorem \thm{mainthm}(iii). We see that the residue of the $\sptcrank\pmod{4}$ divides the marked 
overpartitions of $8$ with even smallest part into $4$ classes, $2$ of which are both
of size $4$ and the other two of which are of sizes $3$ and $4$. This illustrates Theorem
\thm{mainthm}(viii), noting that $8$ is twice a square.
\end{example}

\begin{example}[$n=5$]
There are $20$ marked overpartitions of $5$ with smallest part odd so that $\sptb_1(5)=20$.
$$
\begin{array}{ccccccccc}
(\pi,j) &\pi_1 &\pi_2 &\nu(\pi_1) &k(\pi_2,s(\pi_1))&\kb &\sptcrank &\pmod{4} &\pmod{5}\\
(\ov{4}+1,1) & 1  & 4  & 1 & 0  & 0  & 0  &0 &0\\
(\ov{3}+1+1,1) & 1+1  & 3  & 2 & 0  & 1  & -1 &3   &4\\
(\ov{3}+1+1,2) & 1+1  & 3  & 2 & 0  & 0  & 1  &1 &1\\
(\ov{2}+1+1+1,1) & 1+1+1  & 2  & 3 & 0  & 2  & -2  &2 &3\\
(\ov{2}+1+1+1,2) & 1+1+1  & 2  & 3 & 0  & 1  & -1 &3  &4\\
(\ov{2}+1+1+1,3) & 1+1+1  & 2  & 3 & 0  & 0  & 2 &2  &2\\
(\ov{2}+2+1,1) & 2+1  & 2  & 1 & 0  & 0  & 1  &1 &1\\
(1+1+1+1+1,1) & 1+1+1+1+1  & \dd  & 5 & 0  & 4  & -4 &0  &1\\
(1+1+1+1+1,2) & 1+1+1+1+1  & \dd  & 5 & 0  & 3  & -3 &1  &2\\
(1+1+1+1+1,3) & 1+1+1+1+1  & \dd  & 5 & 0  & 2  & -2 &2  &3\\
(1+1+1+1+1,4) & 1+1+1+1+1  & \dd  & 5 & 0  & 1  & -1 &3  &4\\
(1+1+1+1+1,5) & 1+1+1+1+1  & \dd  & 5 & 0  & 0  & 4  &0 &4\\
(2+1+1+1,1) & 2+1+1+1  & \dd  & 3 & 0  & 2  & -2 &2  &3\\
(2+1+1+1,2) & 2+1+1+1  & \dd  & 3 & 0  & 1  & 0 &0  &0\\
(2+1+1+1,3) & 2+1+1+1  & \dd  & 3 & 0  & 0  & 3 &3  &3\\
(2+2+1,1) & 2+2+1  & \dd  & 1 & 0  & 0  & 2 &2  &2\\
(3+1+1,1) & 3+1+1  & \dd  & 2 & 0  & 1  & 0 &0  &0\\
(3+1+1,2) & 3+1+1  & \dd  & 2 & 0  & 0  & 2 &2  &2\\
(4+1,1) & 4+1  & \dd  & 1 & 0  & 0  & 1 &1  &1\\
(5,1) & 5  & \dd  & 1 & 0  & 0  & 0  &0 &0
\end{array}
$$
We see that the residue of the $\sptcrank\pmod{5}$ divides the marked 
overpartitions of $5$ with odd smallest part into $5$ equal classes. This illustrates
Theorem \thm{mainthm}(v). We see that the residue of the $\sptcrank\pmod{4}$ divides the marked
overpartitions of $5$ with odd smallest part into $4$ classes, $2$ of which are both
of size $4$ and the other two of which are both of size $6$. This illustrates Theorem
\thm{mainthm}(vii), noting that $5$ is not an odd square.
\end{example}

\subsubsection{Proof of the main result} 
\label{subsubsec:proofmain}

\begin{bij}
\mylabel{bij:PHI}
Let $\mathcal{M}$ denote the set of marked overpartitions.
There is 
a weight-preserving bijection 
$$
\Phi\,:\, \mathcal{M} \longrightarrow \SPB
$$
such that,
\beq
\kb(\pi,j) = k(\vec{\lambda}),
\mylabel{eq:kid}
\eeq
and
\beq
\sptcrank(\pi,j) = \crankb(\vec{\lambda}),
\mylabel{eq:sptccid}
\eeq
where
$$
\vec{\lambda} = (\lambda_1,\lambda_2) = \Phi(\pi,j).
$$
\end{bij}
Once this theorem is proved, the main result Theorem \thm{mainthm} 
will then follow from Theorems
\ref{SptBarRankCrank}, \ref{SptBar2RankCrank}, \ref{SptBar1RankCrank}, and
\ref{thm:SPthm} and the appropriate
dissections listed in Section \ref{sec:statementOfResults}.

Before we can construct the bijection $\Phi$,
we need to extend Euler's Theorem that the number of partitions of $n$ into
distinct parts equals the number of partitions of $n$ into odd parts.
Let $n$ be a nonnegative integer. Let $\mathcal{D}_n$ denote the set of
partitions into distinct parts $\ge n+1$. Let $\mathcal{P}_n$ denote the
set of partitions into parts $\ge n+1$ in which all parts $> 2n$ are odd.
Then we have

\begin{theorem}
\mylabel{thm:ExtendEuler}
Let $n\ge0$ and $\ell\ge1$. Then
the number of partitions of $\ell$ from $\mathcal{D}_n$ equals the number of
partitions of $\ell$ from $\mathcal{P}_n$.
\begin{remark}
Euler's Theorem is the case $n=0$.
\end{remark}
\begin{proof}
\begin{align}
& 1 + \sum_{\pi\in\mathcal{D}_n} q^{\abs{\pi}} = \prod_{j=n+1}^\infty (1 + q^j)
\mylabel{eq:proofEE}\\
&= \prod_{j=n+1}^\infty \frac{(1+q^j)(1-q^j)}{(1-q^j)}
= \prod_{j=n+1}^\infty \frac{(1-q^{2j})}{(1-q^j)}
\nonumber\\
&= \frac{1}{(1-q^{n+1})(1-q^{n+2}) \cdots (1-q^{2n}) (q^{2n+1};q^2)_\infty}
\nonumber\\
&= 1 + \sum_{\pi\in\mathcal{P}_n} q^{\abs{\pi}}.                                           
\end{align}
The result follows by considering the coefficient of $q^{\ell}$ on both 
sides of this identity.
\end{proof}
\end{theorem}

We require a bijective proof of this theorem. Glaisher \cite[p.23]{Pak}
has a well-known straightforward bijective proof of Euler's Theorem. We extend
this in a natural way to obtain a bijective proof of our theorem. 

\begin{bij}
Let $n\ge1$. There is a weight-preserving bijection
$$
\Psi_n \,:\, \mathcal{D}_n \longrightarrow \mathcal{P}_n,
$$
$$
\Psi_n(\pi) = \lambda,
$$ 
such that
\beq
k(\pi,n) = \mbox{$\#$ of parts of $\lambda\le 2n-1$}.
\mylabel{eq:Psinprop}
\eeq
\end{bij}

We define the weight-preserving bijection $\Psi_n$ as follows.
Let $\pi\in\mathcal{D}_n$. We describe the image of each part $m$ of $\pi$.
We note that $m\ge n+1$ and as before we write
$$
m = b(m) \, 2^{j(m)},
$$
where $b(m)$ is odd and $j(m)\ge 0$, we note $j(m)\ge j_0(m,n)$. We map
\beq
m \longmapsto \overbrace{2^{j_0(m,n)}b(m), 2^{j_0(m,n)}b(m), \dots, 2^{j_0(m,n)}b(m)}^{\mbox{$2^{j(m)-j_0(m,n)}$ times}},
\mylabel{eq:imm}
\eeq
which preserves the weight since
$$
m = \left(2^{j(m)-j_0(m,n)}\right) \left( 2^{j_0(m,n)}\,b(m)\right).
$$
We recall that $j_0(m,n)$ is the smallest nonnegative integer $j_0$
such that
$$
b(m) 2^{j_0} \ge n+1,
$$
and we see that each image part is $\ge n+1$. If an image part $2^{j_0(m,n)}b(m)$
is even then $j_0(m,n)\ge1$ and
$$
2^{j_0(m,n)}b(m) \le 2n,
$$
as noted before so that each even image part is $\le 2n$. Also any
odd image part $2^{j_0(m,n)}b(m)=b(m)\ge n+1$. This induces a well-defined
map
$$
\Psi_n \,:\, \mathcal{D}_n \longrightarrow \mathcal{P}_n.
$$
We show this map is onto. Let $\lambda$ be a partition in $\mathcal{P}_n$. Let
$p$ be a part of $\lambda$ and let $\mu_p$ denote its multiplicity.
Then we write
$$
p = b(p) 2^{j_0(p,n)} \ge n+1,
$$
we note $j(p)=j_0(p,n)$ since $p$ is a part of $\lambda$ and $\lambda\in\mathcal{P}_n$. 
Now we write $\mu_p$ in
binary
$$
\mu_p = \sum_a 2^{\mu_p(a)}. 
$$
This part $p$ with multiplicity $\mu_p$
arises from a partition in $\mathcal{D}_n$ with parts $b(p) 2^{j_0(p,n) +\mu_p(a)}$
under the action of $\Psi_n$. We see that $\Psi_n$ is onto and 
Theorem \thm{ExtendEuler}
implies that it is a weight-preserving bijection.

Next we prove \eqn{Psinprop}. We let
$$
\ktwid = \mbox{$\#$ of parts $p$ of $\lambda$ where $p\le 2n-1$}.
$$
We note that if $m$ is a part of $\pi$ then as before
$$
m = b(m) 2^{j(m)} \ge n+1,
$$
and $j(m)\ge j_0(m,n)$. Under the map $\Psi_n$ the image of
$m$ is given by \eqn{imm}. This contributes $2^{j(m)-j_0(m,n)}$ to $\ktwid$
provided $b(m) 2^{j_0(m,n)} < 2n$, and \eqn{Psinprop} follows.

\begin{example}[$n=3$]
We illustrate the bijection $\Psi_n$ when $n=3$. There are $6$ partitions of 
$16$ in $\mathcal{D}_3$, the set of partitions into distinct parts $\ge 4$:
$$
\begin{array}{llll}
7+5+4  & \rightarrow\qquad  7\cdot2^0, 5\cdot2^0, 1\cdot2^2 
       &\rightarrow\qquad  7\cdot2^0, 5\cdot2^0, 1\cdot2^2  &\rightarrow 7+5+4 \\
9+7    & \rightarrow\qquad  9\cdot2^0, 7\cdot2^0 
       &\rightarrow\qquad  9\cdot2^0, 7\cdot2^0  &\rightarrow 9+7 \\
10+6    & \rightarrow\qquad  5\cdot2^1, 3\cdot2^1 
       &\rightarrow\qquad  5\cdot2^0, 5\cdot2^0, 3\cdot2^1  &\rightarrow 6+5+5 \\
11+5    & \rightarrow\qquad  11\cdot2^0, 5\cdot2^0 
       &\rightarrow\qquad   11\cdot2^0, 5\cdot2^0 &\rightarrow 11+5 \\
12+4    & \rightarrow\qquad  3\cdot2^2, 1\cdot2^2 
       &\rightarrow\qquad   3\cdot2^1, 3\cdot2^1, 1\cdot2^2 &\rightarrow 6+6+4 \\
16    & \rightarrow\qquad  1\cdot2^4 
       &\rightarrow\qquad   1\cdot2^2, 1\cdot2^2, 1\cdot2^2, 1\cdot2^2 &\rightarrow 
      4+4+4+4 
\end{array}
$$
Each partition has been mapped into $\mathcal{P}_3$, the set of partitions
with smallest part $\ge4$ and all parts $> 6$ are odd.
\end{example}

We are now ready to construct our  weight-preserving bijection 
$\Phi\,:\, \mathcal{M} \longrightarrow \SPB$.  Suppose $(\pi,j)$ is a
marked overpartition with $1\le j \le \nu(\pi)$.
As described before 
we let 
$\pi_1$ be the partition formed by the non-overlined parts of $\pi$, 
$\pi_2$ be the partition (into distinct parts) formed by the overlined parts 
of $\pi$, so that
$$
s(\pi_2) > s(\pi_1) = s(\pi) = n.
$$
We let
\begin{align*}
\pi_1 &= (\overbrace{n,n,\dots,n}^\nu,n_2,n_3,\dots,n_a), \\
\pi_2 &= (m_1, m_2, \dots, m_b),                 
\end{align*}
where
\begin{align*}
 &n<n_2\le n_3\le \cdots \le n_a, \\
 &n<m_1< m_2< \cdots < m_b.     
\end{align*}
Define
\beq
\Phi(\pi,j) = \vec{\lambda} = (\lambda_1,\lambda_2),
\eeq
where
\beq
\begin{array}{ll}
\lambda_1 &= (\overbrace{n,n,\dots,n}^j,n_2,n_3,\dots,n_a), \\
\lambda_2 &= (\overbrace{n,n,\dots,n}^{\nu-j},\Psi_n(\pi_2)).
\end{array}  
\mylabel{eq:lam12def}
\eeq
The map $\Phi$ is clearly weight-preserving. We see that $s(\lambda_1)=n$
and $\lambda_1\in\mathcal{P}$,
In addition, $\Psi_n(\pi_2)$ is a partition into parts $\ge n+1$ with
all parts $\ge 2n+1$ being odd so that $\lambda_2\in\SPB$ and the map
$\Phi$ is well-defined. By \eqn{sptbid} and \eqn{sptbeq} we need only show
that $\Phi$ is onto. 

Let
$$
\vec{\lambda} = (\lambda_1,\lambda_2) \in \SPB.
$$
Let $n=s(\lambda_1)$ so that $\lambda_1$, $\lambda_2\in\mathcal{P}$,
$s(\lambda_2)\ge s(\lambda_1)=n$ and all parts of $\lambda_2$ $\ge 2n+1$
are odd. Let $j=\nu(\lambda_1)$, and let $\ell$ denote the number of
parts of $\lambda_2$ that are equal to $n$, so that $j\ge1$ and $\ell\ge0$.
Remove any parts of $\lambda_2$ equal to $n$ to form the partition
$\lambdatwid_2$ and add the parts removed from $\lambda_2$ to $\lambda_1$
to form the partition $\pi_1$. Now let $\pi_2 = \Psi_n^{-1}(\lambdatwid_2)$
so that $\pi_2$ is a partition into distinct parts $\ge n+1$. Form the partition
$\pi$ by overlining the parts of $\pi_2$ and adding them to $\pi_1$.
We see that $(\pi,j)\in\mathcal{M}$, $1 \le j \le \nu(\pi)=j+\ell$ and
\beq
\Phi(\pi,j) = \vec{\lambda} = (\lambda_1,\lambda_2).
\eeq
The map $\Phi$ is onto and hence a bijection.

    Now we prove \eqn{kid}, \eqn{sptccid}. As before we let
$$
\Phi(\pi,j) = \vec{\lambda} = (\lambda_1,\lambda_2),
$$
where $\lambda_1$, $\lambda_2$ are given in \eqn{lam12def}, so that
$s(\lambda_1)=n$, and $1 \le j \le \nu(\pi)=\nu(\pi_1)$. Then
\begin{align*}
k(\vec{\lambda}) &= \nu - j + 
(\mbox{$\#$ of parts of $\Psi_n(\pi_2) \le 2n-1$}) \\
& = \nu(\pi_1) - j + k(\pi_2,n) \qquad\mbox{(by \eqn{Psinprop})}\\
& = \nu(\pi_1) - j + k(\pi_2,s(\pi_1)) \\
& = \kb(\pi,j),
\end{align*}
which proves \eqn{kid}.
Finally, from \eqn{crankbdef} we have
\beq
\crankb(\vec{\lambda}) = 
\begin{cases}
\mbox{($\#$ of parts of $\pi_1\ge s(\pi_1)+\kb) - \kb$} & \mbox{if $\kb>0$}\\
\mbox{($\#$ of parts of $\pi_1) - 1$} & \mbox{if $\kb=0$,}
\end{cases}
\eeq
where $\kb=\kb(\pi,j)$, since $\kb(\pi,j)=k(\vec{\lambda})$ and if
$\kb=\kb(\pi,j)=0$, then $\nu(\pi_1)=j$ and $k(\pi_2,s(\pi_1))=0$
in which case, the number of parts of $\pi_1$ equals the
number of parts of $\lambda_1$. Hence we have
$$
\crankb(\vec{\lambda}) = \sptcrank(\pi,j),
$$
which is \eqn{sptccid}. This completes the proof of our main result.


\section{Proofs of Theorems \ref{SptBarRankCrank},
\ref{MSptRankCrank},
\ref{SptBar2RankCrank}, \ref{SptBar1RankCrank}
}

These four proofs all follow the same method. The generating function
for the rank series is rewritten using Watson's transformation
and then the two variable series matches the difference of a rank 
and crank by Bailey's Lemma. 

We recall a pair of sequences of functions, $(\alpha_n,\beta_n)$,
forms a Bailey pair for $(a,q)$ if
\begin{align}
	\beta_n &= \sum_{r=0}^n \frac{\alpha_r}
		{\aqprod{q}{q}{n-r}\aqprod{aq}{q}{n+r}}
.
\end{align}
The limiting case of Bailey's Lemma gives for a Bailey pair
$(\alpha_n,\beta_n)$ that
\begin{align}
	\sum_{n=0}^\infty
	\aqprod{\rho_1,\rho_2}{q}{n} \Parans{\frac{aq}{\rho_1\rho_2}}^n  \beta_n
	&=
	\frac{\aqprod{aq/\rho_1,aq/\rho_2}{q}{\infty}}
		{\aqprod{aq,aq/\rho_1\rho_2}{q}{\infty}}
	\sum_{n=0}^\infty
	\frac{\aqprod{\rho_1,\rho_2}{q}{n} (\frac{aq}{\rho_1\rho_2})^n \alpha_n}
		{\aqprod{aq/\rho_1,aq/\rho_2}{q}{n}}
.
\end{align}

\begin{proof}[Proof of Theorem \ref{SptBarRankCrank}]
We use the Bailey pair E(1) of \cite[page 469]{Slater}
for $(1,q)$ given by
\begin{align*}
	\alpha_n &= \PieceTwo{1}{(-1)^n2q^{n^2}}{n=0}{n\ge 1}
	\\
	\beta_n &= \frac{1}{\aqprod{q^2}{q^2}{n}}. 
\end{align*}
Then by Bailey's Lemma we have that
\begin{align*}
	\sum_{n=0}^\infty \frac{\aqprod{z,z^{-1}}{q}{n}q^n}
		{\aqprod{-q,q}{q}{n}}
	&= \frac{\aqprod{zq,z^{-1}q}{q}{\infty}}{\aqprod{q,q}{q}{\infty}}
		\Parans{1+
			\sum_{n=1}^\infty \frac{\aqprod{z,z^{-1}}{q}{n}(-1)^n2q^{n^2+n}}
				{\aqprod{zq,z^{-1}q}{q}{n}}
		}
	\\
	&= \frac{\aqprod{zq,z^{-1}q}{q}{\infty}}{\aqprod{q,q}{q}{\infty}}
		\Parans{1+
			2\sum_{n=1}^\infty \frac{(1-z)(1-z^{-1})(-1)^nq^{n^2+n}}
				{(1-zq^n)(1-z^{-1}q^n)}
		}	
.
\end{align*}

But then
\begin{align*}
	\SB(z,q)
	 =& \sum_{n=1}^\infty \frac{q^n \aqprod{-q^{n+1},q^{n+1}}{q}{\infty}}
		{\aqprod{zq^n,z^{-1}q^n}{q}{\infty}} 
	\\
	=& \frac{\aqprod{-q,q}{q}{\infty}}{\aqprod{z,z^{-1}}{q}{\infty}}
		\cdot \sum_{n=0}^\infty \frac{\aqprod{z,z^{-1}}{q}{n}q^n}
				{\aqprod{-q,q}{q}{n}}
		- \frac{\aqprod{-q,q}{q}{\infty}}{\aqprod{z,z^{-1}}{q}{\infty}}
	\\
	=&
		\frac{\aqprod{-q,q,zq,z^{-1}q}{q}{\infty}}{\aqprod{z,z^{-1},q,q}{q}{\infty}}
			\Parans{1+
				2\sum_{n=1}^\infty \frac{(1-z)(1-z^{-1})(-1)^nq^{n^2+n}}
					{(1-zq^n)(1-z^{-1}q^n)}
			}	 
		- \frac{\aqprod{-q,q}{q}{\infty}}{\aqprod{z,z^{-1}}{q}{\infty}}
	\\
	=&
		\frac{\aqprod{-q}{q}{\infty}}{(1-z)(1-z^{-1})\aqprod{q}{q}{\infty}}
			\Parans{1+
				2\sum_{n=1}^\infty \frac{(1-z)(1-z^{-1})(-1)^nq^{n^2+n}}
					{(1-zq^n)(1-z^{-1}q^n)}
			}	 
		- \frac{\aqprod{q^2}{q^2}{\infty}}{\aqprod{z,z^{-1}}{q}{\infty}}
	\\
	=& \frac{1}{(1-z)(1-z^{-1})}\Parans{
			\sum_{n=0}^\infty\sum_{m=-\infty}^\infty \overline{N}(m,n)z^mq^n
			- \sum_{n=0}^\infty\sum_{m=-\infty}^\infty \overline{M}(m,n)z^mq^n
		}
.
\end{align*}
This proves the theorem.
\end{proof}

\begin{proof}[Proof of Theorem \ref{MSptRankCrank}]
Before using a Bailey pair, we will apply a limiting case of Watson's
transformation to the generating function of $N2(m,n)$. We recall Watson's 
transformation gives
\begin{align}
	&\sum_{n=0}^\infty \frac{\aqprod{aq/bc,d,e}{q}{n}(\frac{aq}{de})^n}
			{\aqprod{q,aq/b,aq/c}{q}{n}}
	\\	
	=&
	\frac{\aqprod{aq/d,aq/e}{q}{\infty}}{\aqprod{aq,aq/de}{q}{\infty}}
	\sum_{n=0}^\infty 
	\frac{\aqprod{a,\sqrt{a}q,-\sqrt{a}q,b,c,d,e}{q}{n}(aq)^{2n}(-1)^nq^{n(n-1)/2}}
		{\aqprod{q,\sqrt{a},-\sqrt{a},aq/b,aq/c,aq/d,aq/e}{q}{n}(bcde)^n}
.
\end{align}
Applying this with $q\mapsto q^2,a=1,b=z,c=z^{-1},d=-q$ and  
$e\rightarrow\infty$ we get the following.
\begin{align}
	& \sum_{n=0}^\infty q^{n^2}
		\frac{\aqprod{-q}{q^2}{n}}{\aqprod{zq^2}{q^2}{n}\aqprod{z^{-1}q^2}{q^2}{n}}
	\\
	&=	\lim_{e\rightarrow\infty}
		\sum_{n=0}^\infty \frac{\aqprod{q^2,-q,e}{q^2}{n}(-1)^ne^{-n}q^{n} }
			{\aqprod{q^2,z^{-1}q^2,zq^2}{q^2}{n}}
	\\
	&= \frac{\aqprod{-q}{q^2}{\infty}}{\aqprod{q^2}{q^2}{\infty}}
		\Parans{1+
			\lim_{a\rightarrow 1, e\rightarrow\infty}
			\sum_{n=1}^\infty
			\frac{(1-a)\aqprod{-q^2,z,z^{-1},e}{q^2}{n} q^{n^2+2n}}			
				{(1-\sqrt{a})\aqprod{-1,z^{-1}q^2,zq^2}{q^2}{n} e^n}	
			}
	\\
	&=
		\frac{\aqprod{-q}{q^2}{\infty}}{\aqprod{q^2}{q^2}{\infty}}
		\Parans{1+
			\sum_{n=1}^\infty
			\frac{(1+q^{2n})(1-z)(1-z^{-1})(-1)^nq^{2n^2+n}}			
				{(1-zq^{2n})(1-z^{-1}q^{2n})}	
			}
.
\end{align}

Using one of the unlabeled Bailey pairs in \cite[page 468]{Slater} 
we have, after replacing $q$ by $q^2$, a Bailey pair for $(1,q^2)$ given by
\begin{align*}
	\alpha_n &= \PieceTwo{1}{(-1)^nq^{2n^2}(q^n+q^{-n})}{n=0}{n\ge 1}
	\\
	\beta_n &= \frac{1}{\aqprod{-q,q^2}{q^2}{n}}. 
\end{align*}
Then by Bailey's Lemma we have that
\begin{align*}
	&\sum_{n=0}^\infty \frac{\aqprod{z,z^{-1}}{q^2}{n}q^{2n}}
		{\aqprod{-q,q^2}{q^2}{n}}
	\\
	&= 
		\frac{\aqprod{zq^2,z^{-1}q^2}{q^2}{\infty}}
			{\aqprod{q^2,q^2}{q^2}{\infty}}
		\Parans{1+
			\sum_{n=1}^\infty \frac{\aqprod{z,z^{-1}}{q^2}{n}(-1)^nq^{2n^2+2n}(q^n+q^{-n})}
				{\aqprod{zq^2,z^{-1}q^2}{q^2}{n}}
		}
	\\
	&= 
		\frac{\aqprod{zq^2,z^{-1}q^2}{q^2}{\infty}}
			{\aqprod{q^2,q^2}{q^2}{\infty}}
		\Parans{1+
			\sum_{n=1}^\infty \frac{(1-z)(1-z^{-1})(-1)^nq^{2n^2+2n}(q^n+q^{-n})}
				{(1-zq^{2n})(1-z^{-1}q^{2n})}
		}	
.
\end{align*}

But then
\begin{align*}
	&\STwoB(z,q)
	\\
	 =& 
		\sum_{n=1}^\infty 
		\frac{q^{2n}\aqprod{q^{2n+2},-q^{2n+1}}{q^2}{\infty}}
			{\aqprod{zq^{2n},z^{-1}q^{2n}}{q^2}{\infty}} 
	\\
	=& 
		\frac{\aqprod{-q,q^2}{q^2}{\infty}}
			{\aqprod{z,z^{-1}}{q^2}{\infty}}
		\cdot 
		\sum_{n=0}^\infty \frac{\aqprod{z,z^{-1}}{q^2}{n}q^{2n}}
			{\aqprod{-q,q^2}{q^2}{n}}
		- \frac{\aqprod{-q,q^2}{q^2}{\infty}}
				{\aqprod{z,z^{-1}}{q^2}{\infty}}
	\\
	=&
		\frac{\aqprod{-q,q^2,zq^2,z^{-1}q^2}{q^2}{\infty}}
			{\aqprod{z,z^{-1},q^2,q^2}{q^2}{\infty}}
		\cdot\Parans{1+
					\sum_{n=1}^\infty \frac{(1-z)(1-z^{-1})(-1)^nq^{2n^2+2n}(q^n+q^{-n})}
						{(1-zq^{2n})(1-z^{-1}q^{2n})}
				}
		- \frac{\aqprod{-q,q^2}{q^2}{\infty}}
				{\aqprod{z,z^{-1}}{q^2}{\infty}}
	\\
	=& \frac{\aqprod{-q}{q^2}{\infty}}{(1-z)(1-z^{-1})\aqprod{q^2}{q^2}{\infty}}
		\Parans{1+
				\sum_{n=1}^\infty \frac{(1-z)(1-z^{-1})(-1)^nq^{2n^2+n}(q^{2n}+1)}
					{(1-zq^{2n})(1-z^{-1}q^{2n})}
		}	 
		\\
		&-\frac{\aqprod{-q,q^2}{q^2}{\infty}}
			{(1-z)(1-z^{-1})\aqprod{zq^2,z^{-1}q^2}{q^2}{\infty}}
	\\
	=& \frac{1}{(1-z)(1-z^{-1})}\Parans{
			\sum_{n=0}^\infty\sum_{m=-\infty}^\infty N2(m,n)z^mq^n
			- \sum_{n=0}^\infty\sum_{m=-\infty}^\infty M2(m,n)z^mq^n
		}
.
\end{align*}
This proves the theorem.

\end{proof}

\begin{proof}[Proof of Theorem \ref{SptBar2RankCrank} ]
We have
\begin{align}
	&\SB_2(z,q)
	\\
	=&
	\sum_{n=1}^\infty
		\frac{q^{2n}\aqprod{-q^{2n+1},q^{2n+1}}{q}{\infty}}
			{\aqprod{zq^{2n},z^{-1}q^{2n}}{q}{\infty}}
	\\
	=&
	\frac{\aqprod{-q,q}{q}{\infty}}{\aqprod{z,z^{-1}}{q}{\infty}}
	\sum_{n=1}^\infty \frac{q^{2n}\aqprod{z,z^{-1}}{q}{2n}}
		{\aqprod{-q,q}{q}{2n}}
	\\
	=&
		\frac{\aqprod{-q,q}{q}{\infty}}{\aqprod{z,z^{-1}}{q}{\infty}}
		\sum_{n=0}^\infty \frac{q^{2n}\aqprod{z,z^{-1}}{q}{2n}}
			{\aqprod{-q,q}{q}{2n}}
		 -\frac{\aqprod{-q,q}{q}{\infty}}{\aqprod{z,z^{-1}}{q}{\infty}} 
.
\end{align}

Using the Bailey pair in proof of Theorem \ref{SptBarRankCrank} 
along with the Bailey pair for $(1,q)$
\begin{align}
	\alpha_n &= \PieceTwo{1}{2(-1)^n}{n=0}{n\ge 1}
	\\
	\beta_n &= \frac{(-1)^n}{\aqprod{-q,q}{q}{n}}
\end{align}
of \cite[page 468]{Slater}, we have the Bailey pair
\begin{align*}
	\alpha_n &= \PieceTwo{1}{(-1)^n(1+q^{n^2})}{n=0}{n\ge 1}
	\\
	\beta_n &= \frac{1}{2\aqprod{q^2}{q^2}{n}}
		+ \frac{(-1)^n}{2\aqprod{q^2}{q^2}{n}}
	\\
	&= \PieceTwo{\frac{1}{\aqprod{q^2}{q^2}{n}}}{0}{n\equiv 0 \pmod{2}}{n\equiv 1 \pmod{2}}
.
\end{align*}
Thus
\begin{align*}
	&\sum_{n=0}^\infty \frac{q^{2n}\aqprod{z,z^{-1}}{q}{2n}}
		{\aqprod{-q,q}{q}{2n}}
	\\	
	=&
 	\sum_{n=0}^\infty \aqprod{z,z^{-1}}{q}{n} q^n \beta_n
	\\
	=&
		\frac{\aqprod{zq,z^{-1}q}{q}{\infty}}{\aqprod{q,q}{q}{\infty}}
 		\sum_{n=0}^\infty \frac{\aqprod{z,z^{-1}}{q}{n} q^n \alpha_n}
			{\aqprod{zq,z^{-1}q}{q}{n}}
	\\
	=&
	\frac{\aqprod{zq,z^{-1}q}{q}{\infty}}{\aqprod{q,q}{q}{\infty}}
	\Parans{1+
		\sum_{n=1}^\infty \frac{(1-z)(1-z^{-1})(-1)^nq^n(1+q^{n^2})}
			{(1-zq^n)(1-z^{-1}q^n)}
	}
.
\end{align*}

And so 
\begin{align}
	&\SB_2(z,q)
	\\
	=& 
		\frac{\aqprod{-q}{q}{\infty}}{(1-z)(1-z^{-1})\aqprod{q}{q}{\infty}}
		\Parans{1+
		\sum_{n=1}^\infty \frac{(1-z)(1-z^{-1})(-1)^nq^n(1+q^{n^2})}
			{(1-zq^n)(1-z^{-1}q^n)}
		}
	\nonumber\\
	&-\frac{\aqprod{q^2}{q^2}{\infty}}{(1-z)(1-z^{-1})\aqprod{zq,z^{-1}q}{q}{\infty}}
.
\end{align}
This proves the theorem.

\end{proof}

\begin{proof}[Proof of Theorem \ref{SptBar1RankCrank}]
With $\SB(z,q)$and $\SB_2(z,q)$ known, we also know $\SB_1(z,q)$. 
However we can also derive the result from a Bailey pair as we have for the other series.

We have
\begin{align*}
	\SB_1(z,q)
	&= \sum_{n=0}^\infty \frac{q^{2n+1}\aqprod{-q^{2n+2},q^{2n+2}}{q}{\infty}}
			{\aqprod{zq^{2n+1},z^{-1}q^{2n+1}}{q}{\infty}}
	\\
	&=
		\frac{\aqprod{-q,q}{q}{\infty}}{\aqprod{z,z^{-1}}{q}{\infty}}
		\sum_{n=0}^\infty \frac{\aqprod{z,z^{-1}}{q}{2n+1}q^{2n+1}}{\aqprod{-q,q}{q}{2n+1}}.	
\end{align*}

By combining Bailey pairs as we did for $\SB_2(z,q)$, we have a Bailey pair for $(1,q)$ given by
\begin{align*}
	\alpha_n &= \PieceTwo{0}{(-1)^n(q^{n^2}-1)}{n=0}{n\ge 1}
	\\
	\beta_n &= \frac{1}{2\aqprod{q^2}{q^2}{n}}
		- \frac{(-1)^n}{2\aqprod{q^2}{q^2}{n}}
	\\
	&= \PieceTwo{0}{\frac{1}{\aqprod{q^2}{q^2}{n}}}{n\equiv 0 \pmod{2}}{n\equiv 1 \pmod{2}}.
\end{align*}
By Bailey's Lemma we have then
\begin{align*}
	&\sum_{n=0}^\infty \frac{\aqprod{z,z^{-1}}{q}{2n+1}q^{2n+1}}{\aqprod{-q,q}{q}{2n+1}}
	\\
	&= \frac{\aqprod{zq,z^{-1}}{q}{\infty}}{\aqprod{q,q}{q}{\infty}}	
		\sum_{n=0}^\infty\frac{\aqprod{z,z^{-1}}{q}{n}q^n\alpha_n}{\aqprod{zq,z^{-1}q}{q}{n}}
	\\
	&= \frac{\aqprod{zq,z^{-1}}{q}{\infty}}{\aqprod{q,q}{q}{\infty}}	
		\sum_{n=1}^\infty\frac{(1-z)(1-z^{-1})q^n(-1)^n(q^{n^2}-1)}
			{(1-zq^n)(1-z^{-1}q^n)}.
\end{align*}

This gives
\begin{align}\label{SeriesForS1Bar}
	\SB_1(z,q)
	&= \frac{\aqprod{-q}{q}{\infty}}{(1-z)(1-z^{-1})\aqprod{q}{q}{\infty}}
		\sum_{n=1}^\infty \frac{(1-z)(1-z^{-1})q^n(-1)^n(q^{n^2}-1)}
			{(1-zq^n)(1-z^{-1}q^n)}
\end{align}
and completes the proof.
\end{proof}

As pointed out by the referee, it is also possible to deduce these identities
from Watson's transformation, rather than from Bailey pairs and Bailey's
Lemma.

\section{Dissections}

\begin{proof}[Proofs of Theorems \ref{ThreeDissectionForDysonRank} and \ref{ThreeDissectionForM2Rank} ]
We are to show
\begin{align}
	\label{ProofOf3ForRankEq1}
	\sum_{n=0}^\infty\sum_{r=0}^2 \overline{N}(r,3,3n)\zeta_3^r q^{n}
	&=
		\frac{\aqprod{q^3}{q^3}{\infty}^4\aqprod{q^2}{q^2}{\infty}}
			{\aqprod{q}{q}{\infty}^2\aqprod{q^6}{q^6}{\infty}^2}
	,
	\\
	\label{ProofOf3ForRankEq2}
	\sum_{n=0}^\infty\sum_{r=0}^2 \overline{N}(r,3,3n+1)\zeta_3^r q^{n}
	&=
		2\frac{\aqprod{q^3}{q^3}{\infty}\aqprod{q^6}{q^6}{\infty}}
			{\aqprod{q}{q}{\infty}}
	,
	\\
	\label{ProofOf3ForM2RankEq1}
	\sum_{n=0}^\infty\sum_{r=0}^2 N2(r,3,3n+1)\zeta_3^r q^{n}
	&=
		\frac{\aqprod{q^6}{q^6}{\infty}^4}
			{\aqprod{q^2}{q^2}{\infty}\aqprod{q^3}{q^3}{\infty}\aqprod{q^{12}}{q^{12}}{\infty}}
.
\end{align}

For (\ref{ProofOf3ForRankEq1}) we have 
\begin{align}\label{prop2Eq1}
	&\sum_{n=0}^\infty
		\Parans{\overline{N}(0,3,3n) + \overline{N}(1,3,3n)\zeta_3
			+\overline{N}(2,3,3n)\zeta_3^2  
		}q^{n}
	\nonumber\\
	&= 
	\sum_{n=0}^\infty
		\Parans{\overline{N}(0,3,3n) - \overline{N}(1,3,3n)  
		}q^{n}
	\nonumber\\
	&= 
	\frac{\aqprod{q^3}{q^3}{\infty}^2\aqprod{-q}{q}{\infty}}
			{\aqprod{q}{q}{\infty}\aqprod{-q^3}{q^3}{\infty}^2}	
	\nonumber\\
	&= 
	\frac{\aqprod{q^3}{q^3}{\infty}^4\aqprod{q^2}{q^2}{\infty}}
			{\aqprod{q}{q}{\infty}^2\aqprod{q^6}{q^6}{\infty}^2}
\end{align}
The penultimate equality in (\ref{prop2Eq1}) is the first part of Theorem 1.1 of \cite{LO1}, although we've omitted their $-1$ term. The $-1$ is due to how one interprets the empty overpartition and its rank. We use the convention that the empty overpartition has rank 0 and don't adjust the $q^0$ term of the generating function.

Equations (\ref{ProofOf3ForRankEq2}) and (\ref{ProofOf3ForM2RankEq1}) are also just restatements of results 
in \cite{LO1} and \cite{LO2}, respectively.
\end{proof}

\begin{proof}[Proofs of Theorems \ref{FiveDissectionForDysonRank} and \ref{FiveDissectionForM2Rank}]
We see we are to prove
\begin{align}
	\label{ProofFiveDissectionForRankEq1}
	\sum_{n=0}^\infty\sum_{k=0}^4\Parans{\overline{N}(k,5,5n)\zeta_5^k}q^n
	&=
		\frac{\aqprod{q^4,q^6}{q^{10}}{\infty}\aqprod{q^5}{q^5}{\infty}^2}
			{\aqprod{q^2,q^3}{q^5}{\infty}^2\aqprod{q^{10}}{q^{10}}{\infty}}
		+
		2(\zeta_5+\zeta_5^{-1})q
			\frac{\aqprod{q^{10}}{q^{10}}{\infty}}
			{\aqprod{q^3,q^4,q^6,q^7}{q^{10}}{\infty}}	
	,
	\\
	\label{ProofFiveDissectionForRankEq2}
	\sum_{n=0}^\infty\sum_{k=0}^4\Parans{\overline{N}(k,5,5n+3)\zeta_5^k}q^n
	&= 
		\frac{2(1-\zeta_5-\zeta_5^{-1})\aqprod{q^{10}}{q^{10}}{\infty}}
			{\aqprod{q^2,q^3}{q^{5}}{\infty}}
	,
	\\
	\label{ProofFiveDissectionForM2RankEq1}
	\sum_{n=0}^\infty\sum_{k=0}^4\Parans{N2(k,5,5n+1)\zeta_5^k}q^n
	&= 
		\frac{\aqprod{-q^5,q^{10}}{q^{10}}{\infty}}
			{\aqprod{q^2,q^8}{q^{10}}{\infty}}
	,
	\\
	\label{ProofFiveDissectionForM2RankEq2}
	\sum_{n=0}^\infty\sum_{k=0}^4\Parans{N2(k,5,5n+3)\zeta_5^k}q^n
	&= 
		(\zeta_5+\zeta_5^4)
		\frac{\aqprod{-q^5,q^{10}}{q^{10}}{\infty}}
			{\aqprod{q^4,q^6}{q^{10}}{\infty}}
.
\end{align}

But we see that
\begin{align}
	& \overline{N}(0,5,5n) 
		+ \overline{N}(1,5,5n)\zeta_5
		+ \overline{N}(2,5,5n)\zeta_5^2
		+ \overline{N}(3,5,5n)\zeta_5^3 
		+ \overline{N}(4,5,5n)\zeta_5^4
	\\
	&=
		\overline{N}(0,5,5n) 
		+ \overline{N}(1,5,5n)(\zeta_5+\zeta_5^4)
		+ \overline{N}(2,5,5n)(\zeta_5^2+\zeta_5^3)
	\\
	&= 
		\overline{N}(0,5,5n) 
		- \overline{N}(2,5,5n) 
		+ (\zeta_5+\zeta_5^4)(\overline{N}(1,5,5n)-\overline{N}(2,5,5n)). 
\end{align}
By the difference formulas in \cite{LO1} we have then
\begin{align}
	\sum_{n=0}^\infty\sum_{k=0}^4\Parans{\overline{N}(k,5,5n)\zeta_5^k}q^n
	&= 
		\frac{\aqprod{-q^2,-q^3}{q^5}{\infty}\aqprod{q^5}{q^5}{\infty}}
			{\aqprod{q^2,q^3}{q^{5}}{\infty}\aqprod{-q^5}{q^5}{\infty}}
		+ \frac{2(\zeta_5+\zeta_5^4)q\aqprod{q^{10}}{q^{10}}{\infty}}
			{\aqprod{q^3,q^4,q^6,q^7}{q^{10}}{\infty}}
	\nonumber\\
	&=
		\frac{\aqprod{q^4,q^6}{q^{10}}{\infty}\aqprod{q^5}{q^5}{\infty}^2}
			{\aqprod{q^2,q^3}{q^5}{\infty}^2\aqprod{q^{10}}{q^{10}}{\infty}}
		+
		2(\zeta_5+\zeta_5^{-1})q
			\frac{\aqprod{q^{10}}{q^{10}}{\infty}}
			{\aqprod{q^3,q^4,q^6,q^7}{q^{10}}{\infty}}	
. 
\end{align}

Equations (\ref{ProofFiveDissectionForRankEq2}), 
(\ref{ProofFiveDissectionForM2RankEq1}), and 
(\ref{ProofFiveDissectionForM2RankEq2}) are also just restatements of the results
in \cite{LO1} and \cite{LO2}.
\end{proof}

\begin{proof}[Proof of Theorem \ref{ThreeDissectionOverpartionCrank} ]
By definition we have
\begin{align} 
	\sum_{n=0}^\infty\sum_{m=-\infty}^\infty \overline{M}(m,n)\zeta_3^m q^{n}	
	&= \frac{\aqprod{q^2}{q^2}{\infty}}
			{\aqprod{\zeta_3q}{q}{\infty}\aqprod{\zeta_3^{-1}q}{q}{\infty}}
	\nonumber\\
	&= 
	\frac{\aqprod{q^2}{q^2}{\infty}\aqprod{q}{q}{\infty}}
			{\aqprod{q^3}{q^3}{\infty}}.
\end{align}
We see we are to show
\begin{align}\label{crankDissectionEq1}
	\frac{\aqprod{q^2}{q^2}{\infty}\aqprod{q}{q}{\infty}}
			{\aqprod{q^3}{q^3}{\infty}}
	&= 
		\frac{\aqprod{q^9}{q^9}{\infty}^4\aqprod{q^6}{q^6}{\infty}}
			{\aqprod{q^3}{q^3}{\infty}^2\aqprod{q^{18}}{q^{18}}{\infty}^2}
		-q\frac{\aqprod{q^{18}}{q^{18}}{\infty}\aqprod{q^9}{q^9}{\infty}}{\aqprod{q^3}{q^3}{\infty}}
		-2q^2\frac{\aqprod{q^{18}}{q^{18}}{\infty}^4}{\aqprod{q^9}{q^9}{\infty}^2\aqprod{q^6}{q^6}{\infty}}
.
\end{align}

Replacing $q$ by $q^{1/3}$ and multiplying by 
$\frac{\aqprod{q}{q}{\infty}}{\aqprod{q^3}{q^3}{\infty}\aqprod{q^6}{q^6}{\infty}}$, 
the proposition is equivalent to 
\begin{align}
	\label{prop3Eq1}
	\frac{\aqprod{q^{1/3}}{q^{1/3}}{\infty}\aqprod{q^{2/3}}{q^{2/3}}{\infty}}
		{\aqprod{q^3}{q^3}{\infty}\aqprod{q^6}{q^6}{\infty}}
	&=
		\frac{\aqprod{q^3}{q^3}{\infty}^3\aqprod{q^2}{q^2}{\infty}}
			{\aqprod{q}{q}{\infty}\aqprod{q^6}{q^6}{\infty}^3}
		-q^{1/3}
		-2q^{2/3}\frac{\aqprod{q}{q}{\infty}\aqprod{q^6}{q^6}{\infty}^3}
			{\aqprod{q^3}{q^3}{\infty}^3\aqprod{q^2}{q^2}{\infty}}
.
\end{align}

If we let $v$ be the infinite continued fraction
\begin{align}
	v = \cfrac{q^{1/3}}
			{
				1+\cfrac{q+q^2}
					{
						1+\cfrac{q^2+q^4}
							{
								1+\cfrac{q^3+q^6}
								{1+\dots}
							}
					}	 
			}
\end{align}
then by Entry 3.3.1(a) of Ramanujan's Lost notebook part I \cite{AB} we have
\begin{align}
	v &= q^{1/3}\frac{\aqprod{q}{q^2}{\infty}}{\aqprod{q^3}{q^6}{\infty}^3}.
\end{align}

Thus with $x(q)=q^{-1/3}v$ we have
\begin{align}
	x(q) &= \frac{\aqprod{q}{q^2}{\infty}}{\aqprod{q^3}{q^6}{\infty}^3}
	\\
	&= \frac{\aqprod{q}{q}{\infty}\aqprod{q^6}{q^6}{\infty}}{\aqprod{q^3}{q^3}{\infty}^3\aqprod{q^2}{q^2}{\infty}}.
\end{align}

But now (\ref{prop3Eq1}) is exactly Theorem 2 of \cite{Chan}.
\end{proof}

\begin{proof}[Proof of Theorem \ref{FiveDissectionOverpartionCrank} ]
We have
\begin{align}
	\sum_{n=0}^\infty\sum_{m=-\infty}^\infty \overline{M}(m,n)\zeta_5^mq^n
	&=
		\frac{\aqprod{q^2}{q^2}{\infty}}
			{\aqprod{\zeta_5q,\zeta_5^{-1}q}{q}{\infty}}
\end{align}
and so we find a dissection for this product.

By Lemma 3.9 of \cite{Garvan1} we have
\begin{align}
	\frac{1}{\aqprod{\zeta_5q,\zeta_5^{-1}q}{q}{\infty}}
	&= \frac{1}{\aqprod{q^5,q^{20}}{q^{25}}{\infty}}
		+ \frac{(\zeta_5+\zeta_5^{-1})q}{\aqprod{q^{10},q^{15}}{q^{25}}{\infty}}.
\end{align}

Replacing $q$ by $q^2$ in Lemma 3.18 in \cite{Garvan1} we have
\begin{align}
	\aqprod{q^2}{q^2}{\infty}
	&=
	\aqprod{q^{50}}{q^{50}}{\infty}
	\Parans{
		\frac{\aqprod{q^{20},q^{30}}{q^{50}}{\infty}}{\aqprod{q^{10},q^{40}}{q^{50}}{\infty}}
		-q^2
		-q^4\frac{\aqprod{q^{10},q^{40}}{q^{50}}{\infty}}{\aqprod{q^{20},q^{30}}{q^{50}}{\infty}}
	}.
\end{align}

Expanding the product of these two expressions then gives the result.
\end{proof}

\begin{proof}[Proof of Theorem \ref{ThreeDissectionM2Crank} ]
We see
\begin{align}
	\sum_{n=0}^\infty\sum_{m=-\infty}^\infty M2(m,n)\zeta_3^mq^n
	&=
	\frac{\aqprod{-q}{q^2}{\infty}\aqprod{q^2}{q^2}{\infty}^2}
		{\aqprod{q^6}{q^6}{\infty}}
.
\end{align}

We are then to show
\begin{align*}
	&\frac{\aqprod{-q}{q^2}{\infty}\aqprod{q^2}{q^2}{\infty}^2}
		{\aqprod{q^6}{q^6}{\infty}}
	\\
	=& 
		\frac{\aqprod{q^{18}}{q^{18}}{\infty}^{10}\aqprod{q^{12}}{q^{12}}{\infty}\aqprod{q^3}{q^3}{\infty}}
				{\aqprod{q^{36}}{q^{36}}{\infty}^4\aqprod{q^9}{q^9}{\infty}^4\aqprod{q^6}{q^6}{\infty}^3}
		+q\frac{\aqprod{q^{18}}{q^{18}}{\infty}^4}
			{\aqprod{q^{36}}{q^{36}}{\infty}\aqprod{q^9}{q^9}{\infty}\aqprod{q^6}{q^6}{\infty}}
		-2q^2\frac{\aqprod{q^{36}}{q^{36}}{\infty}^2\aqprod{q^9}{q^9}{\infty}^2\aqprod{q^6}{q^6}{\infty}}
			{\aqprod{q^{18}}{q^{18}}{\infty}^2\aqprod{q^{12}}{q^{12}}{\infty}\aqprod{q^3}{q^3}{\infty}}
.
\end{align*}

Noting $\aqprod{-q}{q^2}{\infty}\aqprod{q^2}{q^2}{\infty}^2 = \aqprod{-q}{-q}{\infty}\aqprod{q^2}{q^2}{\infty}$, 
in equation (\ref{crankDissectionEq1}) of the proof of Theorem \ref{ThreeDissectionOverpartionCrank} 
we replace $q$ by $-q$ and multiply by 
$\frac{\aqprod{-q^3}{-q^3}{\infty}}{\aqprod{q^6}{q^6}{\infty}}$ to get
\begin{align}
	&\frac{\aqprod{-q}{q^2}{\infty}\aqprod{q^2}{q^2}{\infty}^2}{\aqprod{q^6}{q^6}{\infty}}
	\nonumber
	\\
	\label{ThreeDissectionM2CrankEq1}
	=& 
		\frac{\aqprod{-q^9}{-q^9}{\infty}^4}
			{\aqprod{-q^3}{-q^3}{\infty}\aqprod{q^{18}}{q^{18}}{\infty}^2}
		+q\frac{\aqprod{q^{18}}{q^{18}}{\infty}\aqprod{-q^9}{-q^9}{\infty}}{\aqprod{q^6}{q^6}{\infty}}
		-2q^2\frac{\aqprod{q^{18}}{q^{18}}{\infty}^4\aqprod{-q^3}{-q^3}{\infty}}
			{\aqprod{-q^9}{-q^9}{\infty}^2\aqprod{q^6}{q^6}{\infty}^2}.
\end{align}

But we have
\begin{align}
	\label{ThreeDissectionM2CrankEq2}
	\aqprod{-q^3}{-q^3}{\infty}
	&= \frac{\aqprod{q^6}{q^6}{\infty}^3}{\aqprod{q^{12}}{q^{12}}{\infty}\aqprod{q^3}{q^3}{\infty}}
,
	\\
	\label{ThreeDissectionM2CrankEq3}
	\aqprod{-q^9}{-q^9}{\infty}
	&= \frac{\aqprod{q^{18}}{q^{18}}{\infty}^3}{\aqprod{q^{36}}{q^{36}}{\infty}\aqprod{q^9}{q^9}{\infty}}.
\end{align}

Equations (\ref{ThreeDissectionM2CrankEq2}) and 
(\ref{ThreeDissectionM2CrankEq3}) with 
(\ref{ThreeDissectionM2CrankEq1}) then give the theorem.
\end{proof}

\begin{proof}[Proof of Theorem \ref{FiveDissectionM2Crank}]
We have
\begin{align}
	\sum_{n=0}^\infty\sum_{m=-\infty}^\infty M2(m,n)\zeta_5^mq^n
	&=
		\frac{\aqprod{-q}{q^2}{\infty}\aqprod{q^2}{q^2}{\infty}}
	{\aqprod{\zeta_5q^2,\zeta_5^{-1}q^2}{q^2}{\infty}}
\end{align}
and so we find a dissection for this product.

Replacing $q$ by $q^2$ in Lemma 3.9 of \cite{Garvan1} we have
\begin{align}
	\frac{1}{\aqprod{\zeta_5q^2,\zeta_5^{-1}q^2}{q^2}{\infty}}
	&= \frac{1}{\aqprod{q^{10},q^{40}}{q^{50}}{\infty}}
		+ \frac{(\zeta_5+\zeta_5^{-1})q^2}{\aqprod{q^{20},q^{30}}{q^{50}}{\infty}}.
\end{align}

Next we note that $\aqprod{-q}{q^2}{\infty}\aqprod{q^2}{q^2}{\infty} = \aqprod{-q}{-q}{\infty}$ and so 
replacing $q$ by $-q$ in Lemma 3.18 in \cite{Garvan1} we have
\begin{align}
	&\aqprod{-q}{q^2}{\infty}\aqprod{q^2}{q^2}{\infty}
	\nonumber\\
	=&
	\aqprod{-q^{25}}{-q^{25}}{\infty}
	\Parans{
		\frac{\aqprod{q^{10},-q^{15}}{-q^{25}}{\infty}}{\aqprod{-q^{5},q^{20}}{-q^{25}}{\infty}}
		+q
		-q^2\frac{\aqprod{-q^{5},q^{20}}{-q^{25}}{\infty}}{\aqprod{q^{10},-q^{15}}{-q^{25}}{\infty}}
	}
	\\
	=&\aqprod{-q^{25},q^{50}}{q^{50}}{\infty}
	\Parans{
		\frac{\aqprod{q^{10},-q^{15},-q^{35},q^{40}}{q^{50}}{\infty}}
			{\aqprod{-q^{5},q^{20},q^{30},-q^{45}}{q^{50}}{\infty}}
		+q
		-q^2\frac{\aqprod{-q^{5},q^{20},q^{30},-q^{45}}{q^{50}}{\infty}}
			{\aqprod{q^{10},-q^{15},-q^{35},q^{40}}{q^{50}}{\infty}}
	}.
\end{align}

Multiplying out these two 5-dissections then gives
\begin{align*}
	&\frac{\aqprod{-q}{q^2}{\infty}\aqprod{q^2}{q^2}{\infty}}{\aqprod{\zeta_5q,\zeta_5^{-1}q}{q}{\infty}}
	\\
	=&
		\frac{\aqprod{-q^{15},-q^{25},-q^{35},q^{50}}{q^{50}}{\infty}}
		{\aqprod{-q^{5},q^{20},q^{30},-q^{45}}{q^{50}}{\infty}} 
	\\
	& + q\frac{\aqprod{-q^{25},q^{50}}{q^{50}}{\infty}}
			{\aqprod{q^{10},q^{40}}{q^{50}}{\infty}}
	\\
	&+q^2\Parans{
			(\zeta_5+\zeta_5^4)\frac{\aqprod{q^{10},-q^{15},-q^{25},-q^{35},q^{40},q^{50}}{q^{50}}{\infty}}
				{\aqprod{-q^{5},q^{20},q^{20},q^{30},q^{30},-q^{45}}{q^{50}}{\infty}}
			-\frac{\aqprod{-q^{5},q^{20},-q^{25},q^{30},-q^{45},q^{50}}{q^{50}}{\infty}}
				{\aqprod{q^{10},q^{10},-q^{15},-q^{35},q^{40},q^{40}}{q^{50}}{\infty}}
		}
	\\
	&+q^3(\zeta_5+\zeta_5^4)\frac{\aqprod{-q^{25},q^{50}}{q^{50}}{\infty}}
			{\aqprod{q^{20},q^{30}}{q^{50}}{\infty}}
	\\
	&-q^4(\zeta_5+\zeta_5^4)\frac{\aqprod{-q^{5},-q^{25},-q^{45},q^{50}}{q^{50}}{\infty}}
					{\aqprod{q^{10},-q^{15},-q^{35},q^{40}}{q^{50}}{\infty}}.
\end{align*}
This proves the proposition.
\end{proof}

\begin{proof}[Proof of Theorem \ref{ThreeDissectionExtraSeries} ]

We will use Ramanujan's functions 
\begin{align}
	f(a,b) &= \sum_{k=-\infty}^\infty a^{k(k+1)/2}b^{k(k-1)/2},
	\\
	\phi(q) &= f(q,q) = \sum_{k=-\infty}^\infty q^{k^2}.
\end{align}
By Entry 19 of \cite{Berndt} we have
\begin{align}
	f(a,b) = \aqprod{-a}{ab}{\infty}\aqprod{-b}{ab}{\infty}\aqprod{ab}{ab}{\infty}.
\end{align}

Also we have
\begin{align}\label{productForPhi}
	\phi(-q) &= \frac{\aqprod{q}{q}{\infty}}{\aqprod{-q}{q}{\infty}}
	= \frac{\aqprod{q}{q}{\infty}^2}{\aqprod{q^2}{q^2}{\infty}}.
\end{align}

\begin{proposition}
\begin{align}
	\label{prop1Eq1}
	\sum_{n=1}^\infty \frac{(-1)^nq^n(1-q^n)}{1-q^{3n}}
	&=
	\frac{-1}{6}\Parans{1- \frac{\aqprod{q}{q}{\infty}^6\aqprod{q^6}{q^6}{\infty}}
			{\aqprod{q^2}{q^2}{\infty}^3\aqprod{q^3}{q^3}{\infty}^2} }
	\\
	\label{prop1Eq2}
	&=
	\frac{-1}{6}\Parans{1-\frac{\phi(-q)^3}{\phi(-q^3)}} 
.
\end{align}
\end{proposition}
\begin{proof} As in \cite{Fine} we let 
\begin{align}
	E_r(N;m) &= \sum_{\substack{d\mid N\\d\equiv r \pmod{m}}} 1 
					-\sum_{\substack{d\mid N\\d\equiv -r \pmod{m}}} 1.
\end{align}
Thus 
\begin{align}
	\sum_{N=1}^\infty q^N E_r(N;m)
	&= \sum_{n=1}^\infty\sum_{k=0}^\infty q^{kmn+rn}-q^{kmn+(m-r)n}
	\\
	&= \sum_{n=1}^\infty \frac{q^{rn}-q^{(m-r)n}}{1-q^{mn}}.
\end{align}
Similarly we have
\begin{align}
	\sum_{N=1}^\infty q^N\Parans{E_r(N;m) - 2E_r(N/2;m)}
	&= \sum_{n=1}^\infty \frac{(-1)^n(q^{rn}-q^{(m-r)n})}
		{1-q^{mn}}.
\end{align}

Then equation (\ref{prop1Eq1}) is given by equation (32.64) of \cite{Fine}
and (\ref{prop1Eq2}) follows from (\ref{productForPhi}).
\end{proof}

With this we then have
\begin{align}
	\frac{\aqprod{-q}{q}{\infty}}{\aqprod{q}{q}{\infty}}
	\Parans{
		\frac{1}{2}
		+
		\sum_{n=1}^\infty
			\frac{(1-\zeta_3)(1-\zeta_3^{-1})(-1)^nq^n}
			{(1-\zeta_3q^n)(1-\zeta_3^{-1}q^n)}
	}	
	&= 
		\frac{\aqprod{-q}{q}{\infty}}{\aqprod{q}{q}{\infty}}
		\Parans{
			\frac{1}{2}
			+
			3\sum_{n=1}^\infty 
				\frac{(-1)^nq^n(1-q^n)}{(1-q^{3n})}
		}	
	\\
	&= 
		\frac{\aqprod{-q}{q}{\infty}\phi(-q)^3}
		{2\aqprod{q}{q}{\infty}\phi(-q^3)}
	\\
	&= 
		\frac{\aqprod{q}{q}{\infty}^2\aqprod{-q^3}{q^3}{\infty}}
		{2\aqprod{-q}{q}{\infty}^2\aqprod{q^3}{q^3}{\infty}}		
	\\\label{ProofExtraSeries3DissectionEq1}
	&= 
		\frac{\aqprod{q}{q}{\infty}^2\aqprod{q^6}{q^6}{\infty}}
		{2\aqprod{-q}{q}{\infty}^2\aqprod{q^3}{q^3}{\infty}^2}
.
\end{align}

\begin{proposition}\label{ExtraSeries3DisectionProposition1}
\begin{align}
	\frac{\aqprod{q}{q}{\infty}^2}{\aqprod{-q}{q}{\infty}^2}
	&= 	
		\frac{\aqprod{q^9}{q^9}{\infty}^4}
			{\aqprod{q^{18}}{q^{18}}{\infty}^2}
		-4q\frac{\aqprod{q^{18}}{q^{18}}{\infty}\aqprod{q^9}{q^9}{\infty}\aqprod{q^3}{q^3}{\infty}}
					{\aqprod{q^6}{q^6}{\infty}}
		+4q^2\frac{\aqprod{q^{18}}{q^{18}}{\infty}^4\aqprod{q^3}{q^3}{\infty}^2}
					{\aqprod{q^9}{q^9}{\infty}^2\aqprod{q^6}{q^6}{\infty}^2}.
\end{align}
\end{proposition}
\begin{proof}
By the Corollary (page 49) to Entry 31 of \cite{Berndt}, 
we have
\begin{align}
	\phi(q) &= \phi(q^9) + 2qf(q^3,q^{15}).
\end{align}
Replacing $q$ by $-q$ we find that
\begin{align}
	\phi(-q) 
	&= 
		\frac{\aqprod{q^9}{q^9}{\infty}^2}{\aqprod{q^{18}}{q^{18}}{\infty}}
		-2q\aqprod{q^3,q^{15},q^{18}}{q^{18}}{\infty}
	\nonumber\\
	&= 
		\frac{\aqprod{q^9}{q^9}{\infty}^2}{\aqprod{q^{18}}{q^{18}}{\infty}}
		-2q\frac{\aqprod{q^{18}}{q^{18}}{\infty}^2\aqprod{q^3}{q^{3}}{\infty}}
			{\aqprod{q^9}{q^9}{\infty}\aqprod{q^6}{q^6}{\infty}}
.
\end{align}

Thus
\begin{align}
	\frac{\aqprod{q}{q}{\infty}^2}{\aqprod{-q}{q}{\infty}^2}
	=& \phi(-q)^2
	\nonumber\\
	=&
		\frac{\aqprod{q^9}{q^9}{\infty}^4}{\aqprod{q^{18}}{q^{18}}{\infty}^2}
		-4q\frac{\aqprod{q^{18}}{q^{18}}{\infty}\aqprod{q^9}{q^9}{\infty}\aqprod{q^3}{q^3}{\infty}}
				{\aqprod{q^6}{q^6}{\infty}}
		+4q^2\frac{\aqprod{q^{18}}{q^{18}}{\infty}^4 \aqprod{q^3}{q^3}{\infty}^2}
				{\aqprod{q^9}{q^9}{\infty}^2\aqprod{q^6}{q^6}{\infty}^2}.
\end{align}
This proves the proposition.

\end{proof}

With equation (\ref{ProofExtraSeries3DissectionEq1}) and Proposition 
\ref{ExtraSeries3DisectionProposition1},  we have finished the proof of
Theorem \ref{ThreeDissectionExtraSeries}. 

\end{proof}


\begin{proof}[Proof of Theorem \ref{TheoremTwoDissectionSBars}]
We can determine $\SB(i,q)$, $\SB_1(i,q)$, and $\SB_2(i,q)$ from formulas about $\phi$:
\begin{align}
	\label{REntry25.ii}
	\phi(q)-\phi(-q) &= 4\sum_{n=1}^\infty q^{(2n-1)^2}
,
	\\
	\label{REntry25.iii}
	\phi(q)\phi(-q) &= \phi(-q^2)^2
,
	\\
	\label{REntry8.v}
	\phi(-q^2)^2 &= 1 + 4\sum_{n=1}^\infty \frac{(-1)^nq^{n^2+n}}{1+q^{2n}}
,	
	\\
	\label{REntry8.v.p}
	\phi(-q)^2 &= 1 + 4\sum_{n=1}^\infty \frac{(-1)^nq^{n}}{1+q^{2n}}
.		
\end{align}
These equalities can all be found in Ramanujan's Notebooks part III by Berndt \cite{Berndt}.
Equation (\ref{REntry25.ii}) is by Entry 22.i on page 36,
(\ref{REntry25.iii})  is Entry 25.iii on page 40,
(\ref{REntry8.v}) is Entry 8.v on page 114 with $q$ replaced by $q^2$, and
(\ref{REntry8.v.p}) is on page 116 as part of the proof of Entry 8.v.

As in the proof of Theorem \ref{SptBarRankCrank} we have
\begin{align*}
	\SB(z,q) 
	&=
	\frac{\aqprod{-q}{q}{\infty}}{(1-z)(1-z^{-1})\aqprod{q}{q}{\infty}}
	\Parans{1+
		2\sum_{n=1}^\infty \frac{(1-z)(1-z^{-1})(-1)^nq^{n^2+n}}
			{(1-zq^n)(1-z^{-1}q^n)}
	}	 
	- \frac{\aqprod{q^2}{q^2}{\infty}}{\aqprod{z,z^{-1}}{q}{\infty}}
,
\end{align*}
thus
\begin{align*}
	\SB(i,q) 
	&=
	\frac{\aqprod{-q}{q}{\infty}}{2\aqprod{q}{q}{\infty}}
	\Parans{1+
		4\sum_{n=1}^\infty \frac{(-1)^nq^{n^2+n}}
			{1+q^{2n}}
	}	 
	- \frac{\aqprod{q^2}{q^2}{\infty}}{2\aqprod{-q^2}{q^2}{\infty}}
	\\
	&=
	\frac{\phi(-q^2)^2}{2\phi(-q)} - \frac{\phi(-q^2)}{2}
	\\
	&=
	\frac{\phi(q)}{2} - \frac{\phi(-q^2)}{2}
	\\
	&= \sum_{n=1}^\infty q^{n^2} - \sum_{n=1}^\infty (-1)^n q^{2n^2}
.
\end{align*}

By (\ref{SeriesForS1Bar}) we have
\begin{align}
	\SB_1(i,q)
	=& \frac{1}{\phi(-q)}\sum_{n=1}^\infty\frac{(-1)^nq^n(q^{n^2}-1)}{1+q^{2n}}
	\\
	=& \frac{1}{\phi(-q)}
			\Parans{
				\frac{\phi(-q^2)^2-\phi(-q)^2}{4}
			}
	\\
	=&\frac{\phi(-q)}{\phi(-q)}
			\Parans{
				\frac{\phi(q)-\phi(-q)}{4}
			}
	\\
	=& \sum_{n=1}^\infty q^{(2n-1)^2}
.
\end{align}

Lastly, since $\SB_2(z,q) = \SB(z,q) - \SB_1(z,q)$, we have
\begin{align*}
	\SB_2(i,q) 
	&= 
	\sum_{n=1}^\infty q^{(2n)^2} - \sum_{n=1}^\infty (-1)^n q^{2n^2}
.
\end{align*}
\end{proof}

\begin{proof}[Proof of Theorem \ref{FiveDissectionExtraSeries} ]

To start we set
\begin{align}
	C(\tau) 
	&=
	3 + 10\sum_{n=1}^\infty \frac{(-1)^n(q^n-q^{4n})}{1-q^{5n}}
,
	\\
	D(\tau)
	&=
	1 + 10\sum_{n=1}^\infty \frac{(-1)^n(q^{2n}-q^{3n})}{(1-q^{5n})}.
\end{align}
We claim $C(\tau)$ and $D(\tau)$ are elements of $M_1(\Gamma_1(10))$. 

First we define a primitive Dirichlet character modulo $5$ by
\begin{align}
	\chi_{5}(n) &= 	
	\left\{
   	\begin{array}{lll}
      	1 & \hspace{15pt}\mbox{ if } n\equiv 1 \pmod{5}\\
			i & \hspace{15pt}\mbox{ if } n\equiv 2 \pmod{5}\\
			-i & \hspace{15pt}\mbox{ if } n\equiv 3 \pmod{5}\\
       	-1 & \hspace{15pt}\mbox{ if } n\equiv 4 \pmod{5}\\
			0 & \hspace{15pt}\mbox{ otherwise}.
     	\end{array}
	\right.
\end{align}
We then also have a primitive Dirichlet character given by the conjugate $\overline{\chi_5}$.

As in \cite{Kohlberg} and \cite{Garvan2} we set
\begin{align}
	V_{\chi_5,1}(\tau)
	=& 
	\frac{3+i}{10}
	+\sum_{m=1}^\infty \sum_{n=1}^\infty \chi_5(n)q^{mn}
	\\
	&= \frac{3+i}{10} +\sum_{m=1}^\infty\frac{q^m - q^{4m}}{1-q^{5m}}
			+i\sum_{m=1}^\infty\frac{q^{2m} - q^{3m}}{1-q^{5m}}
\end{align}
and
\begin{align}
	V_{\overline{\chi_5},1}(\tau) 
	=& 
	\frac{3-i}{10}
	+\sum_{m=1}^\infty\sum_{n=1}^\infty \overline{\chi_5}(n)q^{mn}
	\\
	&= \frac{3-i}{10} +\sum_{m=1}^\infty\frac{q^m - q^{4m}}{1-q^{5m}}
			-i\sum_{m=1}^\infty\frac{q^{2m} - q^{3m}}{1-q^{5m}}.
\end{align}
Then $V_{\chi_5,1}(\tau)\in M_1(\Gamma_0(5),\chi_5)$ 
and $V_{\overline{\chi_5},1}(\tau)\in M_1(\Gamma_0(5),\overline{\chi_5})$. 
Thus we have $V_{\chi_5,1}(2\tau)\in M_1(\Gamma_0(10),\chi_5)$ 
and $V_{\overline{\chi_5},1}(2\tau)\in M_1(\Gamma_0(10),\overline{\chi_5})$.
Here $M_k(\Gamma,\chi)$ is the vector space of holomorphic modular forms
of weight $k$ with respect to the subgroup $\Gamma$ of $\Gamma_0$ and with
character $\chi$.

We see
\begin{align}
	C(\tau) &= 5\Parans{
						2V_{\chi_5,1}(2\tau) - V_{\chi_5,1}(\tau) 
						+ 2V_{\overline{\chi_5},1}(2\tau) - V_{\overline{\chi_5},1}(\tau)
					}
,
	\\
	D(\tau) &= -i5\Parans{
						2V_{\chi_5,1}(2\tau) - V_{\chi_5,1}(\tau) 
						- 2V_{\overline{\chi_5},1}(2\tau) + V_{\overline{\chi_5},1}(\tau)
					}
,    
\end{align}
but since the characters are different, we must move from $\Gamma_0$ to $\Gamma_1$. 
That is to say we have
 $C(\tau),D(\tau)\in M_1(\Gamma_1(10))$. Noting $\frac{\eta(2\tau)^2}{\eta(\tau)^4}$ 
is a modular form of weight $-1$ for $\Gamma_1(8)$, we have then 
that $C(\tau)\frac{\eta(2\tau)^2}{\eta(\tau)^4}$ 
and $D(\tau)\frac{\eta(2\tau)^2}{\eta(\tau)^4}$ are
 modular functions with respect to $\Gamma_1(40)$.
By modular function, we mean a modular form of weight zero.

We use the following generalized eta notation as in \cite{Robins},
\begin{align}
	\eta_{\delta,g}(\tau)
	&= e^{\pi i P_2(\frac{g}{\delta})\delta\tau}
		\prod_{\substack{m>0\\ m\equiv g\pmod{\delta} }}(1-q^m)
		\prod_{\substack{m>0\\ m\equiv -g\pmod{\delta} }}(1-q^m)
\end{align}
where
\begin{align}
	P_2(t) &= \CBrackets{t}^2 - \CBrackets{t} + \frac{1}{6}.
\end{align}

So for $g=0$ we have
\begin{align}
	\eta_{\delta,0}(\tau)
	&= q^{\frac{\delta}{12}}\aqprod{q^\delta}{q^\delta}{\infty}^2
	= \eta(\delta\tau)^2
\end{align}
and for $0<g<\delta$ we have
\begin{align}
	\eta_{\delta,g}(\tau)
	&= q^{\frac{P_2(\frac{g}{\delta})\delta}{2}} \aqprod{q^g,q^{\delta-g}}{q^\delta}{\infty}.
\end{align}

\begin{proposition}\label{ExtraSeries5DissectionProp1}
\begin{align}\label{propSomethingEq1}
	C(\tau)\frac{\eta(2\tau)}{\eta(\tau)^2}
	&= C_0(q^5) + qC_1(q^5) + q^2C_2(q^5) + q^3C_3(q^5) + q^4C_4(q^5)
\end{align}
where
\begin{align}
	C_0(q) &= \frac{\aqprod{q^5}{q^5}{\infty}^2\aqprod{q^2}{q^2}{\infty}^2}
					{\aqprod{q^{10}}{q^{10}}{\infty}\aqprod{q}{q}{\infty}^4}
				C(\tau)
,
	\\
	C_1(q) &= -4\frac{\aqprod{q^4,q^6,q^{10}}{q^{10}}{\infty}}
				{\aqprod{q^2,q^8}{q^{10}}{\infty}^2\aqprod{q^3,q^7}{q^{10}}{\infty}}
,
	\\
	C_2(q) &= 2\frac{\aqprod{q^{10}}{q^{10}}{\infty}}
				{\aqprod{q^1,q^9}{q^{10}}{\infty}\aqprod{q^4,q^6}{q^{10}}{\infty}}
,	
	\\
	C_3(q) &= -6\frac{\aqprod{q^{10}}{q^{10}}{\infty}}
				{\aqprod{q^2,q^8}{q^{10}}{\infty}\aqprod{q^3,q^7}{q^{10}}{\infty}}
,
	\\
	C_4(q) &= 2\frac{\aqprod{q^2,q^8,q^{10}}{q^{10}}{\infty}}
				{\aqprod{q,q^9}{q^{10}}{\infty}\aqprod{q^4,q^6}{q^{10}}{\infty}^2} 
.
\end{align}
\end{proposition}
\begin{proof}

Multiplying both sides of (\ref{propSomethingEq1}) by $\frac{\eta(2\tau)}{\eta(\tau)^2}$ and noting the powers of $q$ from $\eta_{\delta,g}$ really do match, we see this proposition is equivalent to
\begin{align}\label{propSomethingEq2}
	C(\tau)\frac{\eta(2\tau)^2}{\eta(\tau)^4}
	=&  
		\frac{\eta(2\tau)\eta(25\tau)^2\eta(10\tau)^2}{\eta(\tau)^2\eta(50\tau)\eta(5\tau)^4}
		C(5\tau)
		\nonumber\\
		&-
		4\frac{\GEta{2}{0}{\tau}^{1/2}\GEta{50}{0}{\tau}^{1/2}\GEta{50}{20}{\tau}}
		{\GEta{1}{0}{\tau}\GEta{50}{10}{\tau}^2\GEta{50}{15}{\tau}}
		\nonumber\\
		&+
		2\frac{\GEta{2}{0}{\tau}^{1/2}\GEta{50}{0}{\tau}^{1/2}}
		{\GEta{1}{0}{\tau}\GEta{50}{5}{\tau}\GEta{50}{20}{\tau}}
		\nonumber\\
		&-	
		6\frac{\GEta{2}{0}{\tau}^{1/2}\GEta{50}{0}{\tau}^{1/2}}
		{\GEta{1}{0}{\tau}\GEta{50}{10}{\tau}\GEta{50}{15}{\tau}}
		\nonumber\\
		&+
		2\frac{\GEta{2}{0}{\tau}^{1/2}\GEta{50}{0}{\tau}^{1/2}\GEta{50}{10}{\tau}}
		{\GEta{1}{0}{\tau}\GEta{50}{5}{\tau}\GEta{50}{20}{\tau}^2}
		\nonumber
.
\end{align}

However we have that 
$\frac{\eta(2\tau)\eta(25\tau)^2\eta(10\tau)^2}{\eta(\tau)^2\eta(50\tau)\eta(5\tau)^4}$ 
is a weight $-1$ modular form for $\Gamma_1(200)$ by Theorem 1.64 of \cite{Ono1} and so 
\begin{align*}
	\frac{\eta(2\tau)\eta(25\tau)^2\eta(10\tau)^2}{\eta(\tau)^2\eta(50\tau)\eta(5\tau)^4}
	C(5\tau)
\end{align*}
is a modular function for $\Gamma_1(200)$. By Theorem 3 of \cite{Robins}, 
the four other generalized eta quotients on the right hand side
of (\ref{propSomethingEq2}) are also modular functions on $\Gamma_1(200)$.

We recall some facts about modular functions as in \cite{Rankin} and use 
the notation in section 20 of \cite{Berndt}. 
Suppose $f$ is a modular function with respect to the congruence subgroup $\Gamma$ 
of $\Gamma_0(1)$. For $A\in\Gamma_0(1)$ we have a cusp given 
by $\zeta=A^{-1}\infty$. The width of the cusp $N=N(\Gamma,\zeta)$ is
given by
\begin{align}
	N(\Gamma,\zeta) &= \min\{k>0:\pm A^{-1}T^kA\in\Gamma\},
\end{align}
where $T$ is the translation matrix 
\begin{align*}
	T &= {\Parans{\begin{array}{cc}
				1&1\\
				0&1
			\end{array}}}
.
\end{align*}

If
\begin{align}
	f(A^{-1}\tau) &= \sum_{m=m_0}^\infty b_mq^{m/N}
\end{align}
and $b_{m_0}\not=0$, then we say $m_0$ is the order of $f$ at
$\zeta$ with respect to $\Gamma$ and we denote this value 
by $Ord_\Gamma(f;\zeta)$. By
$ord(f;\zeta)$ we mean the invariant order of $f$ at $\zeta$
given by
\begin{align}
	ord(f;\zeta) = \frac{Ord_\Gamma(f;\zeta)}{N}.
\end{align}

For $z$ in the upper half plane $\mathcal{H}$, we write 
$ord(f;z)$ for the order of $f$ at $z$ as an analytic function
in $z$. We define the order of $f$ at $z$ with respect to
$\Gamma$ by
\begin{align}
	Ord_\Gamma(f;z) = \frac{ord(f;z)}{m},
\end{align}
where $m$ is the order of $z$ as a fixed point of $\Gamma$.

We then have the well known valence formula for modular functions
as the weight zero case of the valence formula for modular forms, 
which is Theorem 4.1.4 of \cite{Rankin}.
Suppose a subset $\mathcal{F}$ of 
$\mathcal{H}\cup\{\infty\}\cup\mathbb{Q}$ is a fundamental
region for the action of $\Gamma$ along with a complete set of
inequivalent cusps, if $f$ is not the zero 
function then
\begin{align}
	\sum_{z\in\mathcal{F}}Ord_\Gamma(f;z) &=0.
\end{align}

To prove (\ref{propSomethingEq2}), we use the valence formula with $f$ being the difference of the two sides of (\ref{propSomethingEq2}). We note the only poles of $f$ can be at the cusps corresponding to $\Gamma_1(200)$ and 
so
\begin{align}
	\sum_{z\in\mathcal{F}}Ord_\Gamma(f;z)
	&\ge \sum_{\zeta\in C}Ord_\Gamma(f;\zeta)
\end{align}
where $C$ is a set of inequivalent cusps.

But if we have a lower bound on the cusps not equivalent to $\infty$, say
\begin{align}
	\sum_{\substack{\zeta\in C\\ \zeta\not\equiv \infty}}
		Ord_\Gamma(f;\zeta)
	\ge -M,	
\end{align}
and we knew $Ord_\Gamma(f;\infty)>M$, then by the valence formula $f$ must be
identically zero. That is to say to prove (\ref{propSomethingEq2}) we would need only verify the 
$q$-series expansions agree past $q^M$.

Noting $C(\tau)$ is a holomorphic modular form, in terms of getting a lower bound on the sum of the orders, we may ignore it. Using Theorem 4 of \cite{Robins}, we can compute the order of the generalized eta quotients at the cusps. Including $\infty$, there are $336$ inequivalent cusps for $\Gamma_1(200)$. To get a lower bound on the sum of orders at cusps not equivalent to $\infty$, at each cusp we take the minimum order of the six generalized eta quotients in
(\ref{propSomethingEq2}). Using Maple for the calculations, we find
\begin{align}
	\sum_{\substack{\zeta\in C\\ \zeta\not\equiv \infty}}
		Ord_\Gamma(f;\zeta)
	\ge -1840.	
\end{align} 
However we also verify in Maple that $f$ vanishes past $q^{1840}$ and so the equality holds.

\end{proof}

\begin{proposition}
\begin{align}
	D(\tau)\frac{\eta(2\tau)}{\eta(\tau)^2}
	&= D_0(q^5) + qD_1(q^5) + q^2D_2(q^5) + q^3D_3(q^5) + q^4D_4(q^5)
\end{align}
where
\begin{align}
	D_0(q) &= \frac{\aqprod{q^5}{q^5}{\infty}^2\aqprod{q^2}{q^2}{\infty}^2}
					{\aqprod{q^{10}}{q^{10}}{\infty}\aqprod{q}{q}{\infty}^4}
				D(\tau)
,
	\\
	D_1(q) &= 2\frac{\aqprod{q^4,q^6,q^{10}}{q^{10}}{\infty}}
				{\aqprod{q^2,q^8}{q^{10}}{\infty}^2\aqprod{q^3,q^7}{q^{10}}{\infty}}
,
	\\
	D_2(q) &= -6\frac{\aqprod{q^{10}}{q^{10}}{\infty}}
				{\aqprod{q,q^9}{q^{10}}{\infty}\aqprod{q^4,q^6}{q^{10}}{\infty}}
,
	\\
	D_3(q) &= -2\frac{\aqprod{q^{10}}{q^{10}}{\infty}}
				{\aqprod{q^2,q^8}{q^{10}}{\infty}\aqprod{q^3,q^7}{q^{10}}{\infty}}
,
	\\
	D_4(q) &= 4\frac{\aqprod{q^2,q^8,q^{10}}{q^{10}}{\infty}}
				{\aqprod{q,q^9}{q^{10}}{\infty}\aqprod{q^4,q^6}{q^{10}}{\infty}^2} 
.
\end{align}
\end{proposition}
\begin{proof}

Since $D$ is also a weight 1 form for $\Gamma_1(10)$ and these are the same products as in the previous proposition, we also need only verify the corresponding equality between modular functions holds past $q^{1840}$.
This verification is done in Maple.

\end{proof}

\begin{proposition}
\begin{align}
	\Parans{2C(\tau) - D(\tau)}
	\frac{\aqprod{q^2}{q^2}{\infty}^2}{\aqprod{q}{q}{\infty}^4}
	&= 5\frac{\aqprod{q^4,q^6}{q^{10}}{\infty}}{\aqprod{q^2,q^3}{q^5}{\infty}^2}
.
\end{align}
\end{proposition}
\begin{proof}
We see this proposition is equivalent to
\begin{align}\label{someOtherPropEq1}
	\Parans{2C(\tau) - D(\tau)}
	\frac{\eta(2\tau)^2}{\eta(\tau)^4}
	&=
	5\frac{\GEta{10}{4}{\tau}}{\GEta{5}{2}{\tau}^2}.
\end{align}

However we know the left hand side of (\ref{someOtherPropEq1}) to be a modular function for $\Gamma_1(40)$. Using Theorem 3 of \cite{Robins} we find that the right hand side is as well. Comparing the orders at cusps
as we did in the proof of Proposition \ref{ExtraSeries5DissectionProp1}, we find a lower bound 
for the sum of orders at the cusps other than $\infty$ to be $-48$. However we verify in Maple that (\ref{someOtherPropEq1}) holds past $q^{48}$ and so the equality must hold.
\end{proof}

\begin{proposition}
\begin{align}
	\Parans{3D(\tau) - C(\tau)}
	\frac{\aqprod{q^2}{q^2}{\infty}^2}{\aqprod{q}{q}{\infty}^4}
	&= 10\frac{q}
		{\aqprod{q^3,q^4,q^6,q^7}{q^{10}}{\infty}\aqprod{q^5}{q^{10}}{\infty}^2}
.
\end{align}
\end{proposition}
\begin{proof}
We see this proposition is equivalent to
\begin{align}\label{someOther2PropEq1}
	\Parans{3D(\tau) - C(\tau)}
	\frac{\eta(2\tau)^2}{\eta(\tau)^4}
	&=
	10\frac{1}{\GEta{10}{3}{\tau}\GEta{10}{4}{\tau}\GEta{10}{5}{\tau}}.
\end{align}

Again both sides are modular functions for $\Gamma_1(40)$ and taking the minimum of orders gives that we need only verify the equality in (\ref{someOther2PropEq1}) holds past $q^{48}$.
\end{proof}

With these propositions we can complete the proof of 
Theorem \ref{FiveDissectionExtraSeries}. We have
\begin{align}
	& 	
	\frac{\aqprod{-q}{q}{\infty}}{\aqprod{q}{q}{\infty}}
	\Parans{
		\frac{1}{2}
		+
		\sum_{n=1}^\infty \frac{(1-\zeta_5)(1-\zeta_5^{-1})(-1)^nq^{n}}
			{(1-\zeta_5q^n)(1-\zeta_5^{-1}q^n)}
	}
	\nonumber\\
	&=
	\frac{\aqprod{-q}{q}{\infty}}{\aqprod{q}{q}{\infty}}
	\Parans{
		\frac{1}{2}
		+
		\sum_{n=1}^\infty 
			\frac{(1-\zeta_5)(1-\zeta_5^{-1})(-1)^nq^{n}(1-q^n)(1-\zeta_5^2q^n)(1-\zeta_5^3q^n)}
			{(1-q^{5n})}
	}
	\nonumber\\
	&=
	\frac{\aqprod{-q}{q}{\infty}}{\aqprod{q}{q}{\infty}}
	\left( 	
			\frac{1}{2}
			+(3+\zeta_5^2 + \zeta_5^3)
			\sum_{n=1}^\infty \frac{(-1)^nq^{n}}{(1-q^{5n})}
			-(4+3\zeta_5^2 + 3\zeta_5^3)
			\sum_{n=1}^\infty \frac{(-1)^nq^{2n}}{(1-q^{5n})}
			\right.
			\nonumber\\
			&\hspace{60pt}\left.
			+(4+3\zeta_5^2 + 3\zeta_5^3)
			\sum_{n=1}^\infty \frac{(-1)^nq^{3n}}{(1-q^{5n})}
			-(3+\zeta_5^2 + \zeta_5^3)
			\sum_{n=1}^\infty \frac{(-1)^nq^{4n}}{(1-q^{5n})}
			\right)
	\nonumber\\
	&=
	\frac{\aqprod{-q}{q}{\infty}}{\aqprod{q}{q}{\infty}}
	\Parans{
			\frac{1}{2}
			+(3+\zeta_5^2 + \zeta_5^3)
			\sum_{n=1}^\infty \frac{(-1)^n(q^{n}-q^{4n})}{(1-q^{5n})}
			-(4+3\zeta_5^2 + 3\zeta_5^3)
			\sum_{n=1}^\infty \frac{(-1)^n(q^{2n}-q^{3n})}{(1-q^{5n})}
	}
	\nonumber\\
	&=
	\frac{\aqprod{-q}{q}{\infty}}{\aqprod{q}{q}{\infty}}
	\Parans{
		\frac{1}{2}
		- \frac{3(3+\zeta_5^2 + \zeta_5^3)}{10}
		+ \frac{(4+3\zeta_5^2 + 3\zeta_5^3)}{10}
		+ \frac{(3+\zeta_5^2 + \zeta_5^3)}{10}C(\tau)
		- \frac{(4+3\zeta_5^2 + 3\zeta_5^3)}{10}D(\tau)
	}
	\nonumber\\
	&=
	\frac{\aqprod{-q}{q}{\infty}}{10\aqprod{q}{q}{\infty}}
	\Parans{
		(3+\zeta_5^2 + \zeta_5^3)C(\tau)
		-(4+3\zeta_5^2 + 3\zeta_5^3)D(\tau)
	}
	\nonumber\\
	&=
	B_0(q^5) + qB_1(q^5) + q^2B_2(q^5) +q^3B_3(q^5) + q^4B_4(q^5)
\end{align}
where
\begin{align}
	B_0(q)
	&=
		\frac{\aqprod{q^5}{q^5}{\infty}^2\aqprod{q^2}{q^2}{\infty}^2}
		{10\aqprod{q^{10}}{q^{10}}{\infty}\aqprod{q}{q}{\infty}^4}
		\Parans{
			(3+\zeta_5^2 + \zeta_5^3)C(\tau)
			-(4+3\zeta_5^2 + 3\zeta_5^3)D(\tau)
		}
	\nonumber\\
	&=
		\frac{\aqprod{q^5}{q^5}{\infty}^2\aqprod{q^2}{q^2}{\infty}^2}
		{10\aqprod{q^{10}}{q^{10}}{\infty}\aqprod{q}{q}{\infty}^4}
		\Parans{
			(2-\zeta_5 - \zeta_5^{-1})C(\tau)
			-(1-3\zeta_5 - 3\zeta_5^{-1})D(\tau)
		} 
	\nonumber\\
	&=
		\frac{\aqprod{q^5}{q^5}{\infty}^2\aqprod{q^4,q^6}{q^{10}}{\infty}}
		{2\aqprod{q^{10}}{q^{10}}{\infty}\aqprod{q^2,q^3}{q^5}{\infty}^2}
		+
		(\zeta_5+\zeta_5^{-1})
		\frac{q\aqprod{q^5}{q^5}{\infty}^2}
		{\aqprod{q^{10}}{q^{10}}{\infty}\aqprod{q^3,q^4,q^6,q^7}{q^{10}}{\infty}\aqprod{q^5}{q^{10}}{\infty}^2}
,
	\\
	B_1(q)
	&=
		(\zeta_5+\zeta_5^{-1}-1)\frac{\aqprod{q^4,q^6,q^{10}}{q^{10}}{\infty}}
		{\aqprod{q^2,q^8}{q^{10}}{\infty}^2 \aqprod{q^3,q^7}{q^{10}}{\infty}}
,
	\\
	B_2(q)
	&=
		(1-2\zeta_5-2\zeta_5^{-1})
		\frac{\aqprod{q^{10}}{q^{10}}{\infty}}
			{\aqprod{q,q^9}{q^{10}}{\infty}\aqprod{q^4,q^6}{q^{10}}{\infty}}
,
	\\
	B_3(q)
	&=
		-\frac{\aqprod{q^{10}}{q^{10}}{\infty}}
			{\aqprod{q^2,q^8}{q^{10}}{\infty}\aqprod{q^3,q^7}{q^{10}}{\infty}}
,
	\\
	B_4(q)
	&=
		(\zeta_5+\zeta_5^{-1})
		\frac{\aqprod{q^2,q^8,q^{10}}{q^{10}}{\infty}}
			{\aqprod{q,q^9}{q^{10}}{\infty}\aqprod{q^4,q^6}{q^{10}}{\infty}^2}
.
\end{align}
This finished the proof of Theorem \ref{FiveDissectionExtraSeries}.

\end{proof}

\section{Remarks}

In section 3 we proved the coefficients of $\SB(z,q)$, $\SB_1(z,q)$,
and $\SB_2(z,q)$ are nonnegative by showing each summand 
$\frac{q^n\aqprod{-q^{n+1},q^{n+1}}{q}{\infty}}{\aqprod{zq^n,z^{-1}q^n}{q}{\infty}}$
has nonnegative coefficients. Numerical evidence suggests $\STwoB(z,q)$
also has nonnegative coefficients. However, the corresponding individual summands 
for $\STwoB(z,q)$ do not have nonnegative
coefficients themselves. In particular we find the coefficient of $q^{10}$
in
$q^4\aqprod{-q^5,q^6}{q^2}{\infty}/\aqprod{zq^4,z^{-1}q^4}{q^2}{\infty}$
to be $z^{-1}+z-1$. Thus for $S2(z,q)$ a more complicated argument is
required.

\begin{conjecture}
For all $m$ and $n$ we have $N_{\STwoB}(m,n)$ is nonnegative.
\end{conjecture}

Related to the nonnegativity of these coefficients is the difference 
between the first rank and crank moment. If we let $N(m,n)$ denote the 
number of partitions of $n$ with rank $m$ and $M(m,n)$ denote the 
number of partitions of $n$ with crank $m$, then for $k\ge 1$ the
$k$th rank moment $N_k(n)$ and $kth$ crank moment $M_k(n)$ are given by
\begin{align}
	N_k(n) &= \sum_{n\in\mathbb{Z}} m^kN(m,n),
	\\
	M_k(n) &= \sum_{n\in\mathbb{Z}} m^kM(m,n)
.
\end{align}
These rank and crank moments were introduced by Atkin and the first author
in \cite{AG}. To allow for non-trivial odd moments Andrews, Chan, 
and Kim in \cite{ACK} defined the modified rank and crank moments by
\begin{align}
	N_k^+(n) &= \sum_{n=1}^\infty m^kN(m,n),
	\\
	M_k^+(n) &= \sum_{n=1}^\infty m^kM(m,n)
.
\end{align}
In the same paper they proved for all positive integers $n$ that
$M_1^+(n) > N_1^+(n)$. This was done by manipulating the generating function
for $M_1^+(n)- N_1^+(n)$ and carefully grouping the terms in such a way that
it is clear the coefficients are positive. However it turns out that 
$M_1^+(n) - N_1^+(n) = N_{\mbox{\rm S}}(0,n)$, the latter was proved to be nonnegative
in \cite{AGL} and so it is immediate that $M_1^+(n) \ge N_1^+(n)$.

Recently Andrews, Chan, Kim, and Osburn in \cite{ACKO} considered the moments for the rank
and crank of overpartitions,
\begin{align}
	\overline{N}_k^+(n) &= \sum_{n=1}^\infty m^k\overline{N}(m,n),
	\\
	\overline{M}_k^+(n) &= \sum_{n=1}^\infty m^k\overline{M}(m,n)
.
\end{align}
We have the generating functions of these moments given as
\begin{align}
	\overline{N}_k(q) &= \sum_{n=1}^\infty \overline{N}^+_k(n)q^n
	,\\
	\overline{M}_k(q) &= \sum_{n=1}^\infty \overline{M}^+_k(n)q^n
.
\end{align}
In that paper they prove $\overline{M}_1^+(n)>\overline{N}_1^+(n)$. As we'll prove shortly,
it also turns out that $\overline{M}_1^+(n)-\overline{N}_1^+(n)=N_{\SB}(0,n)$.
Thus the nonnegativity of the coefficients of $\SB(z,q)$ gives
$\overline{M}_1^+(n)\ge\overline{N}_1^+(n)$.

To begin we use \cite[equation (7.15)]{Garvan1} that
\begin{align}
	\frac{\aqprod{q}{q}{\infty}}{\aqprod{zq,z^{-1}q}{q}{\infty}}
	&=
	\frac{1}{\aqprod{q}{q}{\infty}}
	\Parans{ 
		1 
		+ 
		\sum_{n=1}^\infty\frac{(1-z)(1-z^{-1})(-1)^nq^{n(n+1)/2}(1+q^n)}
			{(1-zq^n)(1-z^{-1}q^n)}
	}
,
\end{align}
so we have
\begin{align}
	\frac{\aqprod{-q,q}{q}{\infty}}{\aqprod{zq,z^{-1}}{q}{\infty}}
	&=
	\frac{\aqprod{-q}{q}{\infty}}{\aqprod{q}{q}{\infty}}
	\Parans{ 
		1 
		+ 
		\sum_{n=1}^\infty\frac{(1-z)(1-z^{-1})(-1)^nq^{n(n+1)/2}(1+q^n)}
			{(1-zq^n)(1-z^{-1}q^n)}
	}
.
\end{align}
With this we can express $\SB(z,q)$ as follows.
\begin{align}
	\SB(z,q)
	=&
		\frac{\aqprod{-q}{q}{\infty}}{(1-z)(1-z^{-1})\aqprod{q}{q}{\infty}}
		\Parans{1+
			2\sum_{n=1}^\infty \frac{(1-z)(1-z^{-1})(-1)^nq^{n^2+n}}
				{(1-zq^n)(1-z^{-1}q^n)}
		}	 
		- 
		\frac{\aqprod{-q,q}{q}{\infty}}{\aqprod{z,z^{-1}}{q}{\infty}}
	\\
	=&
		\frac{\aqprod{-q}{q}{\infty}}{(1-z)(1-z^{-1})\aqprod{q}{q}{\infty}}
		\Parans{1+
			2\sum_{n=1}^\infty \frac{(1-z)(1-z^{-1})(-1)^nq^{n^2+n}}
				{(1-zq^n)(1-z^{-1}q^n)}
		}
		\nonumber
		\\	 
		&-
		\frac{\aqprod{-q}{q}{\infty}}{(1-z)(1-z^{-1})\aqprod{q}{q}{\infty}}
		\Parans{1+
			\sum_{n=1}^\infty \frac{(1-z)(1-z^{-1})(-1)^nq^{n(n+1)/2}(1+q^n)}
				{(1-zq^n)(1-z^{-1}q^n)}
		}
	\\
	=&
		\frac{\aqprod{-q}{q}{\infty}}{\aqprod{q}{q}{\infty}}
		\sum_{n=1}^\infty \frac{(-1)^{n+1}q^{n(n+1)/2}(1+q^n)}
			{(1-zq^n)(1-z^{-1}q^n)}
		-
		2\frac{\aqprod{-q}{q}{\infty}}{\aqprod{q}{q}{\infty}}		
		\sum_{n=1}^\infty \frac{(-1)^{n+1}q^{n^2+n}}
				{(1-zq^n)(1-z^{-1}q^n)}	 
	\\
	=&
		\frac{\aqprod{-q}{q}{\infty}}{\aqprod{q}{q}{\infty}}
		\sum_{n=1}^\infty \frac{(-1)^{n+1}q^{n(n+1)/2}}{(1-q^n)}
			\Parans{\sum_{m=0}^\infty z^mq^{nm} + \sum_{m=1}^\infty z^{-m}q^{nm}}
		\nonumber
		\\
		&-
		2\frac{\aqprod{-q}{q}{\infty}}{\aqprod{q}{q}{\infty}}
		\sum_{n=1}^\infty \frac{(-1)^{n+1}q^{n^2+n}}{(1-q^{2n})}
			\Parans{\sum_{m=0}^\infty z^mq^{nm} + \sum_{m=1}^\infty z^{-m}q^{nm}}	 	
.
\end{align}
In the last equality we have used that
\begin{align}
	\frac{1-q^{2n}}{(1-zq^n)(1-z^{-1}q^n)}
	&= \frac{1}{1-zq^n} + \frac{1}{1-z^{-1}q^n} - 1
.
\end{align}

But $\sum_{n=0}^\infty N_{\SB}(0,n)q^n$ is the coefficient of $z^0$ in $\SB(z,q)$.
From the above we see that
\begin{align}\label{MomentsToSptEq}
	\sum_{n=0}^\infty N_{\SB}(0,n)q^n
	&= 
		\frac{\aqprod{-q}{q}{\infty}}{\aqprod{q}{q}{\infty}}
		\sum_{n=1}^\infty \frac{(-1)^{n+1}q^{n(n+1)/2}}{(1-q^n)}
		-
		2\frac{\aqprod{-q}{q}{\infty}}{\aqprod{q}{q}{\infty}}
		\sum_{n=1}^\infty \frac{(-1)^{n+1}q^{n^2+n}}{(1-q^{2n})}	 		
.
\end{align}
With (\ref{MomentsToSptEq}) and Proposition 2.1 of \cite{ACKO} we have
\begin{align}
	\overline{M}_1(q) - \overline{N}_1(q)
	&= \sum_{n=1}^\infty N_{\SB}(0,n)q^n
.
\end{align}
As explained earlier, we know each $N_{\SB}(0,n)$ to be nonnegative and 
so this is another proof that $\overline{M}_1^+(n)\ge\overline{N}_1^+(n)$.

There is also the $d=e=1$ case for the general spt function, which as
 noted in \cite{BLO1} reduces to $\overline{pp}(n)/4$, where
$\overline{pp}(n)$ is the number of overpartition pairs of $n$. 
The methods in this paper do not give a new proof of the congruences for 
$\overline{pp}(n)$. Using Bailey's lemma on a two variable generating 
function and applying Watson's transformation to the generating function
for the rank of 
overpartition pairs does at first appear to give a difference between 
the rank of overpartition pairs and some residual crank. However, 
the resulting crank is
\begin{align}
	\frac{\aqprod{-q}{q}{\infty}^2}{\aqprod{zq,z^{-1}q}{q}{\infty}},
\end{align}
which can be written in terms of the rank for overpartition pairs 
as in equation (2.1) of \cite{BL1}. In particular, the generating function 
for the rank of overpartition pairs is
\begin{align}
	\sum_{n=0}^\infty\sum_{m=-\infty}^\infty \overline{NN}(m,n)z^mq^n
	&=
	\sum_{n=0}^\infty \frac{\aqprod{-1,-1}{q}{n}q^n}{\aqprod{zq,z^{-1}q}{q}{n}}
\end{align}
and
\begin{align}
	\frac{4}{(1+z)(1+z^{-1})}
	+ \sum_{n=1}^\infty \frac{\aqprod{-1,-1}{q}{n}q^n}{\aqprod{zq,z^{-1}q}{q}{n}}
	&=
	\frac{4\aqprod{-q}{q}{\infty}^2}{\aqprod{zq,z^{-1}q}{q}{\infty}}.
\end{align}
Thus the method of proving congruences in this paper only gives the proofs already given
by Bringmann and Lovejoy in \cite{BL1}.

\bibliographystyle{abbrv}
\bibliography{ranksAndCranksRef}

\end{document}